\documentclass[12pt, twoside]{article}
\usepackage{verbatim, version}
\usepackage{amsmath,amssymb,latexsym, amsthm}
\usepackage{graphicx, subfigure, epsfig, psfrag}
\usepackage{amsfonts, amssymb, dsfont, mathrsfs, color}
\usepackage{diagbox}
\usepackage{paralist,bbding,pifont}
\usepackage{hyperref}

\newlength{\tmarg}
\newlength{\bmarg}
\newlength{\lmarg}
\newlength{\rmarg}

\newtheorem{lemma}{Lemma}[section]
\newtheorem{remark}{Remark}[section]

\newtheorem{proposition}{Proposition}[section]
\newtheorem{definition}{Definition}[section]
\newtheorem{theorem}{Theorem}[section]

\begin{document}
\include{macros}
\title{Weak solutions to the steady compressible Euler equations with source terms}

\author{
  Anxiang Huang\thanks{School of Mathematical Sciences and Institute of Natural Sciences, Shanghai Jiao Tong University, Minhang, Shanghai, 200240, P. R. China ({\tt huanganx@sjtu.edu.cn}).}
}

%\date{\today}
\date{}

\maketitle
%------------------------------------------------------------------------------ 

%============================================================================== 

\begin{abstract}
\quad In this paper, we showed that for some given suitable density and pressure, there exist infinitely many compactly supported solutions with prescribed energy profile. The proof is mainly based on the convex integration scheme. We construct suitable subsolutions and localized plane-wave solutions to the reformulated system, and weak solutions are obtained by iterating these subsolutions.
\end{abstract}


\pagestyle{myheadings}

\thispagestyle{plain}

\markboth{}{}

%------------------------------------------------------------------------------ 

%============================================================================== 
\section{Introduction}
\quad In this paper, we consider the steady Euler equations for an isentropic compressible fluid with source terms:
\begin{align}\label{eq1}
	\begin{cases}
		\text{div}(\rho u)=0,\\
		\text{div}(\rho u\otimes u)+\nabla p(\rho)=\mathbf{B}(\rho u),
	\end{cases}
\end{align}
where $x\in \Omega$ with $\Omega\subset \mathbb{R}^n$ or $\mathbb{T}^n$($n\geq 3$), $\rho:\Omega \rightarrow \mathbb{R}$ denotes the density of the fluid (or gas), and $u:\Omega \rightarrow \mathbb{R}^n$ denotes the velocity, $p=p(\rho)$ denotes the pressure, which is assumed continuously differentiable with $p'(\rho)>0$ on $(0,\infty)$, and $\mathbf{B}$ is a $n\times n$ constant matrix. A pair $(\rho,m)\in L^\infty(\Omega;(0,\infty)\times\mathbb{R}^n)$ is called a weak solution of \eqref{eq1} if
\begin{align*}
    \begin{cases}
         \int_{\Omega} \dfrac{m\otimes m}{\rho} \colon \nabla \Phi +\Phi\cdot \mathbf{B}m dx=0,\\
        \int_{\Omega} m\cdot\nabla \psi dx=0
     \end{cases}
\end{align*}
for every test function $\Phi,\psi\in C^\infty(\Omega;\mathbb{R}^n)$ with $\text{div }\Phi=0$.
 
\quad The Euler system expresses the conservation laws of mass and momentum, and the weak solutions of Euler equations in the sense of distributions have been studied in the theory of hyperbolic conservation laws. The well-posedness of the Euler system is one of the key issues which attract lots of interests. So far, most of well-posedness results need additional conditions (\cite{MR2547977,MR3280249}) since the Riemann problem is difficult to resolve in high dimension.

\quad On the other hand, there are some studies focus on the non-uniqueness of the solutions. The first non-uniqueness result to the Euler equations traces back to Scheffer (\cite{MR1231007}), who constructed a non-trivial weak solution of the incompressible Euler equations in two space dimensions with compact support in space-time. Later, the non-uniqueness for three-dimensional flows was established by Shnirelman (\cite{MR1476315}). De Lellis and Sz\'{e}kelyhidi showed the existence of infinitely many bounded weak solutions to the incompressible Euler solutions (\cite{MR2600877}), yielding an alternative proof of Scheffer's non-uniqueness results. After this groundbreaking work, a series of studies (\cite{MR3374958,MR3330471,MR3090182,MR3254331}), which mainly based on the same method, proved the existence and non-uniqueness of energy dissipative solutions with H\"{o}lder exponent $\alpha<1/3$. Finally, the remaining part of the proof to the Onsager's conjecture is accomplished in \cite{MR3896021,MR3866888}.

\quad All these constructions are based on the method of convex integration. The convex integration method provides an iteration scheme, which passes from a subsolution to infinitely many exact solutions with prescribed energy profile. The iteration scheme starts from a subsolution, adds perturbations in the velocity which oscillate at higher and higher frequencies in each step, producing an irregular limit that is shown to be a weak solution to the Euler system. The solutions obtained by convex integration are sometimes called wild solutions. Except for the incompressible Euler equations, the convex integration method contributes to showing the non-uniqueness of the wild solutions to some other systems, such as MHD equations (\cite{MR4198715,MR4328510,MR4362549}), SQG equations (\cite{MR3987721,MR4126319,MR4340931}), transported equations (\cite{MR3884855,MR4244826,MR4199851}) and so on.

\quad The non-uniqueness of weak solutions of the compressible isentropic Euler equations was also proved by De Lellis and Sz\'{e}kelyhidi (\cite{MR2564474}). They showed that there exist bounded initial data for which there are infinitely many weak solutions for the barotropic Euler equations. This first result on non-uniqueness for the compressible Euler equations was improved in \cite{MR3352460,MR3816641,MR3459023}. So far, most of these constructions by convex integration method for compressible Euler equations are mainly based on the incompressible scheme (\cite{MR4223016,MR3261301,MR4591578}), that is, the oscillations are only added to the velocity (or momentum), and the density remains fixed along the iteration. This leads to solutions with irregular momentum but regular densities. 

\quad In order to study the sharp non-uniqueness for the Euler equation, one possible approach is to find the largest space for the solutions which ensure the non-uniqueness. Stationary solutions are a class of solutions with special structures. In this paper, we consider the $L^\infty$ weak solutions to $\eqref{eq1}$. Our main result is the following:

\begin{theorem}\label{tm1.1}
    Let $n\geq 3$,
    \begin{itemize}
        \item [(a)]{let $\Omega=\mathbb{T}^n$, $\rho\in C(\Omega)$, then for any given continuous function $q=q(x)>0$, there exist infinitely many weak solutions $u\in L^\infty(\Omega;\mathbb{R}^n)$ to $\eqref{eq1}$ such that $\left|u\right|^2=\frac{nq}{\rho}$;}
        \item [(b)]{let $\Omega$ be an open bounded subset of $\mathbb{R}^n$, $\mathbf{B}=a\mathbf{I}_n$ for some $a>0$, $\rho\in C^1(\mathbb{R}^n)$ satisfies $\rho=\bar{\rho}>0$ for $x\in \mathbb{R}^n\backslash\Omega$, and $\int_{\Omega}p(\rho(x))dx=p(\bar{\rho})\left|\Omega\right|$. Then there exists a function $q=q(x)$ such that \eqref{eq1} has infinitely many compact supported weak solutions $u\in L^\infty(\mathbb{R}^n;\mathbb{R}^n)$ such that $\left|u\right|^2=\frac{nq}{\rho}$;}
        \item [(c)]{let $\Omega$ be an open bounded subset of $\mathbb{R}^n$, $\rho(x)=\bar{\rho}>0$ for $x\in \mathbb{R}^n\backslash\Omega$, and $0<\rho(x)<\bar{\rho}$ for $x\in \Omega$. Then there exists a function $q=q(x)$ such that \eqref{eq1} has infinitely many compact supported weak solutions $u \in L^\infty(\mathbb{R}^n;\mathbb{R}^n)$ such that $\left|u\right|^2=\frac{nq}{\rho}$.}
    \end{itemize}
\end{theorem}

\quad There are a few remarks in order.

\begin{remark}
    In \cite{MR3505175}, Choffrut and Sz\'{e}kelyhidi shows that the analogue of the h-principle for incompressible Euler system in $L^\infty$ also holds in the steady case. Theorem \ref{tm1.1}(a) can be regarded as an approach on the analogue of the h-principle for compressible Euler system, though it is not easy to construct a smooth stationary flow $(\rho,u)$ for system \eqref{eq1}.
\end{remark}

\begin{remark}
    The convex integration scheme in this paper can also be adapted to general densities. As mentioned above, for some given subsolutions and energy profile, the perturbations only add to the velocity and the density remains fixed. Hence, one can use this scheme to show that when the density is piecewise constant, there exist infinitely many weak stationary flows $(\rho,u)\in L^\infty$ with certain energy profiles.
\end{remark}

\begin{remark}
    As for compactly supported solutions, the main difficulty is to construct the strict subsolutions with compact supports. Hence we need more restrictions to the density. Since for compressible Euler system, the pressure $p$ is a function of the density, i.e. $p=p(\rho)$, the two restrictions in Theorem \ref{tm1.1}(b) and (c) in fact are both restrictions on densities.
\end{remark}

\quad The key ideas of this paper are as follows. Inspired by \cite{MR2600877}, we first reformulate the compressible Euler system as a differential inclusion, and define the wave-cone and relaxation set. Since the source term may not be zero, we adapt the localized plane-waves which first established by Luo, Xie and Xin (\cite{MR3459023}). Using the localized plane-waves as the building block, we can then construct weak solutions from strict subsolutions by iteration. Finally, we need to construct suitable strict subsolutions. In the case that $\Omega=\mathbb{T}^n$, the strict subsolution can be constructed by a divergence equation. In the case that $\Omega\subset\mathbb{R}^n$, inspired by \cite{MR4207175}, we employ a specific form, which enables us to construct a compactly supported solution for the Poisson equation and serves as our strict subsolution.

\quad The rest of the paper is organized as follows. In Section 2, we introduce the geometric setup and reformulate the Euler system as a differential inclusion. After that, we construct the localized plane waves, which constitute the basic of the convex integration scheme. In Section 3, we use the convex integration method, construct exact solutions from strict subsolution by adding localized plane waves. In Section 4, we construct suitable strict subsolutions for the convex integration. Finally, the proof of the Theorem \ref{tm1.1} are provided in Section 5.

%==============================================================

\section{Geometric Setup}

\quad Inspired by De Lellis and Sz\'{e}kelyhidi, the proof of our main results are mainly based on the convex integration scheme. In this section, we reformulate the system \eqref{eq1} as a differential inclusion, and introduce the geometric concepts we need in the iteration scheme.

\quad Denote the set of symmetric, trace-free $n\times n$ matrices by
\begin{align*}
    \mathbb{S}_0^{n\times n}=\{\mathbf{A}\in\mathbb{R}^{n\times n}:\mathbf{A}^T=\mathbf{A}, \text{tr}(\mathbf{A})=0\}.
\end{align*}

For given constants $a,b>0$, denote
\begin{align*}
    K_{a,b}=\{(m,\mathbf{U})\in\mathbb{R}^n\times\mathbb{S}_0^{n\times n}:\dfrac{m\otimes m}{a}-\mathbf{U}=b\mathbf{I}_n\}.
\end{align*}

\quad A bounded weak solution $(\rho,u)$ to the system
\begin{align*}
\begin{cases}
	\nabla\cdot(\rho u)=0, \\ 
	\nabla\cdot(\rho u\otimes u)+\nabla p(\rho)=\mathbf{B}(\rho u),
\end{cases}
\end{align*}
is equivalent to a bounded quadruple $(\rho,m,\mathbf{U},q)$ satisfying
\begin{align}\label{eq2}
	\begin{cases}
	\nabla\cdot m=0, \\ 
	\nabla\cdot \mathbf{U}+\nabla(p(\rho)+q)=\mathbf{B}m,
	\end{cases}
\end{align}
in the sense of distribution under the constraint
\begin{align*}
	(m,\mathbf{U})(x)\in K_{\rho(x),q(x)}\quad\text{for a.e.}\quad x\in\mathbb{R}^n.
\end{align*}
\quad Now, one can define the subsolutions as follows.
\begin{definition}	The quadruple $\left(\rho,m,\mathbf{U},q\right)\in L^{\infty}(\Omega;(0,+\infty)\times\mathbb{R}^n\times\mathbb{S}_0^{n\times n}\times \mathbb{R})$ is called a subsolution of the system \eqref{eq1} if it solves \eqref{eq2} in the sense of distribution and satisfies
	\begin{align}\label{eq3}
		\left(m,\mathbf{U}\right)(x)\in K_{\rho(x),q(x)}^{\text{co}}.
	\end{align}
Furthermore, since the constrain set $K_{\rho,q}$ is dependent on $x$, for a family of compact sets $K_x\subset K_{\rho(x),q(x)}$ such that the map $x\mapsto K_x$ is continuous in an open set $\mathcal{D}$ in the Hausdorff distance, a subsolution $(\rho,m,\mathbf{U},q)$ is said to be strict in $\mathcal{D}$ with constrain sets $K_x$ if
\begin{align*}
    (\rho,m,\mathbf{U},q)\in C_{\text{loc}}(\mathbb{R}^n),\quad (\rho,m,\mathbf{U},q)|_{\mathcal{D}}\in C^0(\mathcal{D}),
\end{align*}
and
\begin{align*}
    (m,\mathbf{U})(x)\in \mathring{K}_{x}^{\text{co}},\quad \rho(x)>0, q(x)>0\quad\text{for}\quad x\in\mathcal{D}.
\end{align*}
\end{definition}

\quad Compare with the incompressible case, the pressure $p$ is no longer an arbitrary rapidly decreasing function. Hence, the state variables here are quadruple $(\rho,m,\mathbf{U},q)$. In order to formulate the convex hull of $K_{\rho,q}$, we need the following lemma, which is a counterpart in compressible case as \cite[Lemma 3.2]{MR2564474}.

\begin{lemma}\label{lm2.1}
    For $\mathbf{A}\in\mathbb{S}_0^{n\times n}$, let $\lambda_{\text{max}}(\mathbf{A})$ be the largest eigenvalue of $\mathbf{A}$. For $(\rho,m,\mathbf{U})\in (0,\infty)\times\mathbb{R}^n\times\mathbb{S}_0^{n\times n}$, let
    \begin{align*}
        e(\rho,m,\mathbf{U}):=\frac{d}{2}\lambda_{\text{max}}\left(\dfrac{m\otimes m}{\rho}-\mathbf{U}\right),
    \end{align*}
    then
    \begin{itemize}
        \item [(i)]{$e:(0,+\infty)\times \mathbb{R}^n\times\mathbb{S}_0^{n\times n}\rightarrow\mathbb{R}$ is convex;}
        \item [(ii)]{$\dfrac{1}{2}\dfrac{\left|m\right|^2}{\rho}\leq e(\rho,m,\mathbf{U})$, the equality holds if and only if $\mathbf{U}=\dfrac{m\otimes m}{\rho}-\dfrac{1}{n}\dfrac{\left|m\right|^2}{\rho}\mathbf{I}_n$;}
        \item [(iii)]{$\left|\mathbf{U}\right|_{\infty}\leq 2\dfrac{n-1}{n}e(\rho,m,\mathbf{U})$, where $\left|\cdot\right|_{\infty}$ denotes the matrix operator norm;}
        \item [(iv)]{let $\rho>0$ be a given function, then the convex hull of the set\begin{align*}
            K_{\rho,r}:=\left\{(m,\mathbf{U})\in\mathbb{R}^n\times\mathbb{S}_0^{n\times n}:\left|m\right|=r,\mathbf{U}=\frac{m\otimes m}{\rho}-\frac{1}{\rho}\frac{\left|m\right|^2}{n}\mathbf{I}_n\right\}
        \end{align*} can be formulate as\begin{align*}
            K_{\rho,r}^{\text{co}}=\left\{(m,\mathbf{U})\in\mathbb{R}^n\times\mathbb{S}_0^{n\times n}\colon e(\rho,m,\mathbf{U})\leq \frac{r^2}{2\rho}\right\};
        \end{align*}}
        \item [(v)]{if $(m,\mathbf{U})\in\mathbb{R}^n\times\mathbb{S}_0^{n\times n}$, then $r=\sqrt{2\rho e(\rho,m,\mathbf{U})}$ is the smallest $r$ such that $(m,\mathbf{U})\in K_{\rho,r}^{\text{co}}$.}
    \end{itemize}
\end{lemma}
\begin{proof}
    \text{(i).} By definition
    \begin{align}\label{eq4}\notag
        e(\rho,m,\mathbf{U})&=\frac{n}{2}\mathop{\max}\limits_{\xi\in S^{n-1}}\left\langle \xi,\left(\frac{m\otimes m}{\rho}-\mathbf{U}\right)\xi\right\rangle=\frac{n}{2}\mathop{\max}\limits_{\xi\in S^{n-1}}\left\langle \xi,\left\langle \xi,\frac{m}{\sqrt{\rho}}\right\rangle\frac{m}{\sqrt{\rho}}-\mathbf{U}\xi\right\rangle\\
        &=\frac{n}{2}\mathop{\max}\limits_{\xi\in S^{n-1}}\left[\left|\left\langle\xi,\frac{m}{\sqrt{\rho}}\right\rangle\right|^2-\langle \xi,\mathbf{U}\xi\rangle\right].
    \end{align}
    Note that, $\rho\mapsto\left|\langle \xi,\dfrac{m}{\sqrt{\rho}}\rangle\right|^2-\langle\xi,\mathbf{U}\xi\rangle$ is convex with respect to $x$ for any $\xi\in S^{n-1}$, then $e(\rho,m,\mathbf{U})$ is convex.

    \quad\text{(ii).} Since $\dfrac{m\otimes m}{\rho}=\left[\dfrac{m\otimes m}{\rho}-\dfrac{1}{n}\dfrac{\left|m\right|^2}{\rho}\mathbf{I}_d\right]+\dfrac{1}{n}\dfrac{\left|m\right|^2}{\rho}\mathbf{I}_d$, then by definition,
    \begin{align}\label{eq5}\notag
        e(\rho,m,\mathbf{U})&=\frac{n}{2}\mathop{\max}\limits_{\xi\in S^{n-1}}\left\langle\xi,\left(\left[\frac{m\otimes m}{\rho}-\frac{1}{n}\frac{\left|m\right|^2}{\rho}\mathbf{I}_d\right]-\mathbf{U}\right)\xi\right\rangle+\frac{\left|m\right|^2}{2\rho}\\
        &=\frac{n}{2}\lambda_{\text{max}}\left(\left[\frac{m\otimes m}{\rho}-\frac{1}{n}\frac{\left|m\right|^2}{\rho}\mathbf{I}_n\right]\right)+\frac{\left|m\right|^2}{2\rho}.
    \end{align}
    Since $\dfrac{m\otimes m}{\rho}-\dfrac{1}{n}\dfrac{\left|m\right|^2}{\rho}\mathbf{I}_n$ is trace-free and the sum of the eigenvalues to any trace-free matrix equals to zero, one has $\lambda_{\text{max}}(\rho,m,\mathbf{U})\geq 0$, where the equality holds if and only if $\mathbf{U}=\dfrac{m\otimes m}{\rho}-\dfrac{1}{n}\dfrac{\left|m\right|^2}{\rho}\mathbf{I}_n$.

    \quad\text{(iii).} By \eqref{eq4} and \eqref{eq5} one has
    \begin{align*}
        e(\rho,m,\mathbf{U})\geq \frac{n}{2}\mathop{\max}\limits_{\xi\in S^{n-1}}\left(-\langle\xi,\mathbf{U}\xi\rangle\right)=-\frac{n}{2}\lambda_{\text{min}}(\mathbf{U}).
    \end{align*}
    Hence, $-\lambda_{\text{max}}(\mathbf{U})\leq\dfrac{2}{n}e(\rho,m,\mathbf{U})$. Since $\mathbf{U}$ is trace-free, then
    \begin{align*}
        \left|\mathbf{U}\right|_{\infty}\leq(n-1)\left|\lambda_{\text{max}}(\mathbf{U})\right|\leq\dfrac{2(n-1)}{n}e(\rho,m,\mathbf{U}).
    \end{align*}

    \quad\text{(iv).} Without loss of generality, assume that $r=1$, let
    \begin{align*}
        S_1=\left\{(\rho,m,\mathbf{U})\in\mathbb{R}\times\mathbb{R}^n\times\mathbb{S}_0^{n\times n}:e(\rho,m,u)\leq\frac{1}{2\rho}\right\}.
    \end{align*}
    By definition, for any $(\rho,m,\mathbf{U})\in L_{\rho,1}$, $e(\rho,m,\mathbf{U})=\dfrac{1}{2}$. By (i), $e(\rho,m,\mathbf{U})$ is convex, then
    \begin{align*}
        L_{\rho,1}^{\text{co}}\subset S_1.
    \end{align*}
    \quad To show $S_1\subset L_{\rho,1}^{\text{co}}$, note that by (i), (ii) and (iii), $S_1$ is convex and compact, hence $S_1$ equals to the convex hull of its extreme points. It suffices to show that all the extreme points of $S_1$ are contained in $L_{\rho,1}$.

    \quad Assume that $(\rho,m,\mathbf{U})\in S_1\backslash L_{\rho,1}$, by rotation of the coordinates, one can ensure that $\dfrac{m\otimes m}{\rho}-\mathbf{U}$ is a diagonal matrix, where the diagonal terms satisfy $\dfrac{1}{n\rho}\geq \lambda_1\geq\cdots\geq\lambda_n$. Note that $(\rho,m,\mathbf{U})\notin L_{\rho,1}$ implies that $\lambda_n<\dfrac{1}{n\rho}$. In fact, if $\lambda_n=\dfrac{1}{n\rho}$, by (iii), one has $\mathbf{U}=\dfrac{m\otimes m}{\rho}-\dfrac{1}{n\rho}\mathbf{I}_n$. Since $\mathbf{U}$ is trace-free, then $\left|m\right|^2=1$. Thus $\mathbf{U}=\dfrac{m\otimes m}{\rho}-\dfrac{\left|m\right|^2}{n\rho}\text{I}_n$, that is $(\rho,m,\mathbf{U})\in L_{\rho,1}$, which leads to a contradiction.

    \quad Let $e_1,e_2,\cdots,e_n$ be the unit vectors to the corresponding coordinates. Set $m=\mathop{\sum}\limits_{i}m^ie_i$, consider $(\rho,\bar{m},\bar{\mathbf{U}})$ as
    \begin{align*}
        \bar{m}=\dfrac{e_n}{\rho},\quad \bar{\mathbf{U}}=\dfrac{1}{\rho}\mathop{\sum}\limits_{i=1}^{n-1}m^i(e_i\otimes e_n+e_n\otimes e_i).
    \end{align*}
    By direct computation yields
    \begin{align*}
        \dfrac{(m+t\bar{m})\otimes (m+t\bar{m})}{\rho}-(\mathbf{U}+t\bar{\mathbf{U}})=(\dfrac{m\otimes m}{\rho}-\mathbf{U})+\dfrac{1}{\rho^2}(2tm^n+t^2)e_n\otimes e_n.
    \end{align*}
    Since $\lambda_n<\dfrac{1}{n}$, and $e(\rho,m+t\bar{m},\mathbf{U}+t\bar{\mathbf{U}})\leq \dfrac{1}{n}$ holds for sufficiently small $t$, then $(\rho,m,\mathbf{U})+(0,t\bar{m},t\bar{\mathbf{U}})\in L_{\rho,1}$, which implies that $(\rho,m,\mathbf{U})$ is not extreme point of $L_{\rho,1}$. Hence, $S_1\subset L_{\rho,1}^{\text{co}}$.

    \quad\text{(v).} By (iv), one has
    \begin{align*}
         K_{\rho,r}^{\text{co}}=\left\{(m,\mathbf{U})\in\mathbb{R}^n\times\mathbb{S}_0^{n\times n}\colon e(\rho,m,\mathbf{U})\leq \frac{r^2}{2\rho}\right\}.
    \end{align*}
    Hence, if $(m,\mathbf{U})\in \mathbb{R}^n\times S_0^{n\times n}$ and $(m,\mathbf{U})\in K_{\rho,r}^{\text{co}}$, then $r^2\geq 2\rho e(\rho,m,\mathbf{U})$. Hence, $r=\sqrt{2\rho e(\rho,m,U)}$ is the smallest $r$ such that $(m,\mathbf{U})\in K_{\rho,r}^{\text{co}}$. The proof of the lemma is completed.
\end{proof}

\quad Thanks to the Lemma \ref{lm2.1}, we can now formulate the convex hull of the constraint set $K_{\rho,q}$ as
\begin{align*}
    K_{\bar{\rho},\bar{q}}^{\text{co}}=\{(\bar{n},\bar{\mathbf{V}})\in\mathbb{R}^n\times \mathbb{S}_0^{n\times n}:\bar{n}\otimes \bar{n}-\bar{\rho}\bar{\mathbf{V}}\leq\bar{\rho}\bar{q}\mathbf{I}_n\}.
\end{align*}

\quad We are now able to define the wave-cone and the corresponding plane-wave solutions to \eqref{eq2}.

\quad Define
\begin{align*}
    \mathcal{L}(n,\mathbf{V}):=(\nabla\cdot n,\nabla\cdot \mathbf{V}-\mathbf{B}n)^t.
\end{align*}
Observe that, if $(\rho,m,\mathbf{U},q)$ solves \eqref{eq2}, and $(n,\mathbf{V})$ satisfies
\begin{align}\label{eq6}
    \mathcal{L}(n,\mathbf{V})=0,
\end{align}
then $(\rho,m+n,\mathbf{U}+\mathbf{V},q)$ also solves \eqref{eq2}.

By this observation, we can construct the localized plane-wave solutions which form the building blocks for the iteration scheme. 

\quad We define the wave cone to the system \eqref{eq2} as
\begin{align}\label{eq7}
	\Lambda:=\{(\bar{m},\bar{\mathbf{U}})\in\mathbb{R}^{n}\times\mathbb{S}_0^{n\times n}:\text{there exists an } \eta\in\mathbb{R}^{n}\backslash\{0\} \text{ s.t. }\bar{m}\cdot \eta=0,\ \bar{\mathbf{U}}\eta=0\}.
\end{align}
Under this definition, we have the following lemma.
\begin{lemma} For $(\bar{m},\bar{\mathbf{U}})\in\Lambda$, let $h\in C^\infty(\mathbb{R})$ be any continuous differentiable function, then $(m,\mathbf{U}):=(\bar{m},\bar{\mathbf{U}})h(\eta\cdot x)$ is a plane-wave solution to 
\begin{align*}
		\nabla \cdot \bar{m}=0,\quad \nabla \cdot \bar{\mathbf{U}}=0.
\end{align*}
\end{lemma}\label{lm2.2}
\begin{proof}
	The straightforward computations yield
	\begin{align*}
		\nabla\cdot m&=\bar{m}\cdot\nabla(h(x\cdot \eta))=h'(x\cdot \eta)(\bar{m}\cdot \eta)=0,\\
		\nabla \cdot \mathbf{U}&=\bar{\mathbf{U}}\cdot\nabla(h(x\cdot \eta))=h'(x\cdot \eta)(\bar{\mathbf{U}}\eta)=0.
	\end{align*}
	Hence the proof of the lemma is completed.
\end{proof}

\begin{definition} 
For any given $r>0$, we call a line segment $\sigma\subset\mathbb{R}^n\times \mathbb{S}_0^{n\times n}$ admissible if
\begin{itemize}\label{lm2.3}
	\item {$\sigma$ is contained in the interior of $K_r^{\text{co}}$.}
	\item {$\sigma$ is parallel to $(a,a\otimes a)-(b,b\otimes b)$ for some $a,b\in\mathbb{R}^n$ with $\left|a\right|^2=\left|b\right|^2=r$ and $b\neq \pm a$.}
\end{itemize}
\end{definition}

\quad When the dimension $n\geq 3$, admissible line segments are in $\Lambda-$directions, that is, parallel to some $(\bar{m},\bar{\mathbf{U}})\in\Lambda$, which provides plane-wave solutions for our iteration scheme. 

\begin{lemma}\label{lm2.3} Let $w=(\rho,m,\mathbf{U},q)$ with $(m,\mathbf{U})\in \mathring{K}_{\rho,q}^{\text{co}}$, then there exist $\bar{m}\in\mathbb{R}^{n}$, $\bar{\mathbf{U}}\in\mathbb{S}_0^{n\times n}$ such that the admissible segment
\begin{align*}
    \sigma_m&=[(\rho,m-\bar{m},\mathbf{U}-\bar{\mathbf{U}},q),(\rho,m+\bar{m},\mathbf{U}+\bar{\mathbf{U}},q)]\\
    &=\{t(\rho,m-\bar{m},u-\bar{u},q)+(1-t)(\rho,m+\bar{m},u+\bar{u},q):t\in[0,1]\}
\end{align*}
satisfies $[(m-\bar{m},\mathbf{U}-\bar{\mathbf{U}}),(m+\bar{m},\mathbf{U}+\bar{\mathbf{U}})]\in\mathring{K}_{\rho,q}^{co}$ and parallel to 
\begin{align*}
	(0,a,\frac{a\otimes a}{\rho},0)-(0,b,\frac{b\otimes b}{\rho},0)
\end{align*}
for some $a,b\in\mathbb{R}^{n}$ with $\left|a\right|=\left|b\right|=\sqrt{n\rho q}$. Moreover,
\begin{align*}
	\left|\bar{m}\right|\geq \frac{C}{\sqrt{n\rho q}}(n\rho q-\left|m\right|^2).
\end{align*}

\begin{proof}
Denote $z:=(m,\mathbf{U})\in\mathring{K}_{\rho,q}^{co}$, it follows from the Carath\'{e}odory theorem that $z$ can be expressed as $z=\mathop{\sum}\limits_{i=1}^{N+1}\lambda_iz_i$, where $\lambda_i\in(0,1)$, $\sum\limits_{i=1}^{N+1} \lambda_i=1$, and $z_i=(m_i,\dfrac{m_i\otimes m_i}{\rho}-\dfrac{1}{n\rho}\left|m_i\right|^2\mathbf{I}_n)$ for some $\left|m_i\right|=\sqrt{n\rho q}$ and $m_i\neq \pm m_j$, for any $i,j\leq N+1$.
	
\quad Assume that $\lambda_i$ are in ordered, i.e. $\lambda_1\geq \lambda_2\geq\lambda_3\geq\cdots\geq\lambda_{N+1}$, then one has
\begin{align*}
	\left|m-m_1\right|\leq N\cdot\underset{2\leq i\leq N+1}{\max}\lambda_i\left|m_i-m_1\right|.
\end{align*}
Let $j>1$ be such that
\begin{align*}
\lambda_j\left|m_j-m_1\right|=\underset{2\leq i\leq N+1}{\max}\lambda_i\left|m_i-m_1\right|.
\end{align*}
Let 
\begin{align*}
(\bar{m},\bar{\mathbf{U}})&=\frac{1}{2}\lambda_j(z_j-z_1)\\
&=\frac{1}{2}\lambda_j\left(m_j-m_1,\left(\frac{m_j\otimes m_j}{\rho}-\frac{1}{n\rho}\left|m_j\right|^2\mathbf{I}_n\right)-\left(\frac{m_1\otimes m_1}{\rho}-\frac{1}{n\rho}\left|m_1\right|^2\mathbf{I}_n\right)\right).
\end{align*}
Since $\left|m_i\right|=\sqrt{n\rho q}$, then
\begin{align*}
    (\frac{m_j\otimes m_j}{\rho}-\frac{1}{n\rho}\left|m_j\right|^2\mathbf{I}_n)-(\frac{m_1\otimes m_1}{\rho}-\frac{1}{n\rho}\left|m_1\right|^2\mathbf{I}_n)=\frac{m_j\otimes m_j}{\rho}-\frac{m_1\otimes m_1}{\rho}.
\end{align*}
\quad We can now define the segment $\sigma_m$ as
\begin{align*}
	\sigma_m:=[(\rho,z-\frac{\lambda_j}{2}(z_j-z_1),q),(\rho,z+\frac{\lambda_j}{2}(z_j-z_1),q)],
\end{align*}
moreover,
\begin{align*}
	\left|\bar{m}\right|\geq\frac{1}{2N}\left|m-m_1\right|&\geq\frac{1}{2N}(\sqrt{n\rho q}-\left|m\right|)\geq \frac{1}{2N}\frac{(\sqrt{n\rho q}-\left|m\right|)(\sqrt{n\rho q}+\left|m\right|)}{2\sqrt{n\rho q}}\\
	&=\frac{1}{4\sqrt{n\rho q}\cdot N}(n\rho q-\left|m\right|^2).
\end{align*}
This finishes the proof of the lemma.
\end{proof}
\end{lemma}

\begin{remark}
    Compared with \cite[Lemma 4.3]{MR2600877}, Lemma \ref{lm2.3} can be regarded as the property of the admissible segments in compressible case. Though the density is fixed along the iteration scheme, the spatial dimension $n\geq 3$ ensures the wave cone big enough, thus, we can still find plane-wave solutions in $\Lambda-$direction. 
\end{remark}

\quad We next show that, under this definition of wave-cone, we can obtain plane-wave solutions in $\Lambda-$directions:
\begin{lemma}\label{lm2.4} Let $n\geq 3$, for any $w',\ w''\in K_{\bar{\rho},\bar{q}}$ with $w'\neq w''$, we have $w'-w''\in\Lambda$.
\end{lemma}

\begin{proof}
For $(\bar{n}',\bar{\mathbf{V}}')$, $(\bar{n}'',\bar{\mathbf{V}}'')\in K_{\bar{\rho},\bar{q}}$, set $(\bar{n},\bar{\mathbf{V}})=(\bar{n}',\bar{\mathbf{V}}')-(\bar{n}'',\bar{\mathbf{V}}'')$. Let $\xi\in\mathbb{R}^n\backslash\{0\}$. Since $n\geq 3$, there exists a $\xi$ such that $\xi\cdot \bar{n}'=\xi\cdot\bar{n}''=0$. Hence, we have $\xi\cdot \bar{n}=\xi \cdot(\bar{n}'-\bar{n}'')=0$. Furthermore, it holds that
\begin{align*}
\bar{\mathbf{V}}\cdot \xi=\frac{1}{\bar{\rho}}(\bar{n}'\otimes\bar{n}'-\bar{n}''\otimes\bar{n}'')\cdot \xi=\frac{1}{\bar{\rho}}\xi\cdot\bar{n}'(\bar{n}'-\bar{n}'')=0.
\end{align*}
This finishes the proof of the lemma.
\end{proof}

\quad When the source term $\mathbf{B}$ is non-zero, for $w_1,w_2\in K_{\bar{\rho},\bar{q}}$, there may not be plane-wave solutions in $\Lambda-$direction. In \cite{MR3459023}, Luo, Xie and Xin proposed the localized plane wave. They observed that in high-frequency regime, localized plane waves with non-trivial sources terms can be constructed as perturbations of the plane waves without source term. Inspired by their construction, one can construct localized plane-waves in steady case. The next lemma is devoted to correct the errors from the source terms.

\begin{lemma}\label{lm2.5}(\cite[Lemma 8]{MR3459023})
For any $f\in C_c^{\infty}(\mathbb{R}^n;\mathbb{R}^n)$, there exists a $\mathcal{R}[f]\in C^{\infty}(\mathbb{R}^n;\mathbb{S}_0^{n\times n})$ satisfying
\begin{align*}
	\nabla \cdot \mathcal{R}[f]=f.
\end{align*}
Furthermore, $\mathcal{R}$ satisfies the following properties:
\begin{enumerate}[(1)]
	\item [(i)] {$\mathcal{R}$ is a linear operator;}
	\item [(ii)] {$\mathcal{R}[\Delta^2f]$ is a linear combination of third order derivatives of $f$;}
	\item [(iii)] {$\text{supp }\mathcal{R}[\Delta^2f]\subset \text{supp } f$ and $\int_{\mathbb{R}^n}\mathcal{R}[\Delta^2f]dx=0.$;}
	\item [(iv)] {there exists a constant $0<\alpha<1$ such that\begin{align*} \|\mathcal{R}[f]\|_{C^{\alpha}(\mathbb{R}^n)}\leq C\max(\|f\|_{L^{1}(\mathbb{R}^n)},\|f\|_{L^{\infty}(\mathbb{R}^n)}),\end{align*}
	
	where the constants $C$ and $\alpha$ depend only on dimension $n$.}
\end{enumerate}
\end{lemma}	

We are now ready to construct the localized plane-wave solutions to \eqref{eq2}.

\begin{lemma}\label{lm2.6}
Suppose that $w,w_1,w_2\in\mathbb{R}^n\times\mathbb{S}_0^{n\times n}$ satisfy
\begin{align*}
	w=\mu_1w_1+\mu_2w_2, \ \mu_1+\mu_2=1, \ \bar{w}=w_2-w_1\in\Lambda.
\end{align*}

Given an open set $\mathcal{O}$ and $\varepsilon>0$, there exists a $\tilde{w}=(\tilde{n},\tilde{\mathbf{V}})\in C_c^{\infty}(\mathcal{O})$ satisfying the following properties:
\begin{enumerate}[(a)]
	\item [(i)] {$\int_{\mathcal{O}}\tilde{w}(x)dx=0$, $\nabla\cdot\tilde{n}=0$ and $\nabla\cdot\tilde{\mathbf{V}}-\mathbf{B}\tilde n=0$;}
	\item [(ii)] {$\text{dist}(w+\tilde{w}(x),[w_1,w_2])<\varepsilon$ for $x\in \mathcal{O}$;}
	\item [(iii)] {there exist two disjoint open sets $\mathcal{O}_i\subset \mathcal{O}$ such that for $i=1,2$,$$|w+\tilde{w}(x)-w_i|<\varepsilon\ \text{for}\ x\in\mathcal{O}_i;\quad  |\mathcal{H}^{n}(\mathcal{O}_i)-\mu_i\mathcal{H}^{n}(\mathcal{O})|<\varepsilon.$$}
\end{enumerate}
\end{lemma}

\begin{proof}
\quad Let $\delta$ be a small positive constant and $\phi\in C_c^{\infty}(\mathcal{O})$ be a smooth cut-off function satisfying
\begin{align*}
	0\leq\phi\leq 1\quad\text{and}\quad \mathcal{H}^n\left(\{x\in\mathcal{O}:\phi(x)\neq 1\}\right)<\delta.
\end{align*}
Let
\begin{align*}
	h(s)=\begin{cases}
		-\mu_2,\ \ \ &s\in(0,\mu_1],\\
		\mu_1,\ \ \ &s\in(\mu_1,1].
	\end{cases}
\end{align*}
Clearly, one has $\int_0^1h(s)ds=-\mu_2 \mu_1+\mu_1(1-\mu_1)=0$. 

\quad Extend $h$ as a periodic function with period $1$ and let $h_0$ be a smooth approximation of $h$ satisfying
\begin{align*}
	-\mu_2\leq h_0\leq \mu_1,\quad\mathcal{H}^1\left(\{s\in[0,1]:h(s)\neq h_0(s)\}\right)<\delta,\quad\int_0^1h_0(s)ds=0.
\end{align*}
Define $h_k$ by induction as
\begin{align*}
	\tilde{h}_{k+1}(s)=\int_0^sh_k(t)dt,\quad h_{k+1}(s)=\tilde{h}_{k+1}(s)-\int_0^1\tilde{h}_{k+1}(t)dt.
\end{align*}
Hence, for $k\geq 0$, one has
\begin{align*}
	\dfrac{d^jh_k}{ds^j}=h_{k-j},\quad  \int_0^1h_k(s)ds=\int_0^1\left(\tilde{h}_k(s)-\int_0^1\tilde{h}_k(t)dt\right)ds=0,
\end{align*}
and
\begin{align*}
	\|h_k\|_{L^{\infty}}\leq\|h_{k-1}\|_{L^{\infty}}\leq\cdots\leq\|h_0\|_{L^{\infty}}.
\end{align*}
\quad Denote $\bar{w}=\left(\bar{n},\bar{\mathbf{V}}\right)$, and let $\xi\in\mathbb{R}^n\backslash\{0\}$ be the associated vector with $\bar{w}\in\Lambda$. For a given large constant $\lambda\in\mathbb{R}$, let
\begin{align*}
	n'=\lambda^{-6}\Delta^3[\bar{n}h_6(\lambda\xi\cdot x)\phi],\quad \mathbf{V}'=\lambda^{-6}\Delta^3[\bar{\mathbf{V}}h_6(\lambda\xi\cdot x)\phi],
\end{align*}
and
\begin{align*}
	n''=-\nabla\Delta^{-1}\nabla \cdot n',\quad \mathbf{V}''=\mathcal{R}[\mathbf{B}(n'+n'')-\nabla\cdot \mathbf{V}'].
\end{align*}
\quad Define
\begin{align*}
	\tilde{w}=(n,\mathbf{V})=(n'+n'',\mathbf{V}'+\mathbf{V}'').
\end{align*}
We aim to show that $\tilde{w}$ is the function which satisfies the properties (i)-(iii) in Lemma \ref{lm2.6}.

\quad The straightforward computations show that
\begin{align*}
	\nabla\cdot n=\nabla\cdot(n'+n'')=\nabla\cdot n'-\nabla\cdot\nabla\Delta^{-1}\nabla\cdot n'=0,
\end{align*}
and
\begin{align*}
	\nabla\cdot \mathbf{V}=\nabla \cdot \mathbf{V}'+\nabla\cdot\mathcal{R}[\mathbf{B}n-\nabla\cdot \mathbf{V}']=\nabla \cdot \mathbf{V}'+\mathbf{B}n-\nabla\cdot \mathbf{V}'=\mathbf{B}n.
\end{align*}
Hence $(n,\mathbf{V})$ satisfies the condition (i). Furthermore, we have
\begin{align*}
    n'&=\bar{n}h_0(\lambda\xi\cdot x)\phi+\lambda^{-6}\bar{n}\underset{\left|\beta\right|\geq 1,\left|\alpha+\beta\right|=6}{\Sigma}C_{\alpha,\beta}\partial_x^\alpha h_6(\lambda\xi\cdot x)\partial_x^\beta \phi(x)\notag\\
	&=\bar{n}h_0(\lambda\xi\cdot x)\phi+\lambda^{-6}\underset{\left|\beta\right|\geq 1,\left|\alpha+\beta\right|=6}{\Sigma}\lambda^{\left|\alpha\right|}C_{\alpha,\beta}h_{6-\left|\alpha\right|}(\lambda\xi\cdot x)\partial_x^\beta \phi(x),
\end{align*}
and
\begin{align*}
	\mathbf{V}'=\bar{\mathbf{V}}h_0(\lambda\xi\cdot x)\phi+\lambda^{-6}\bar{\mathbf{V}}\underset{\left|\beta\right|\geq 1,\left|\alpha+\beta\right|=6}{\Sigma}\lambda^{\left|\alpha\right|}C_{\alpha,\beta}\xi^{\alpha}h_{6-\left|\alpha\right|}(\lambda\xi\cdot x)\partial_x^\beta\phi(x).
\end{align*}
\quad Note that $\text{supp}\ (n',\mathbf{V}')\subset \text{supp}\ \phi$ and $\int_{\mathcal{O}}(n',\mathbf{V}')dx=0$. Thus,
\begin{align*}
	\|(n',\mathbf{V}')-(\bar{n},\bar{\mathbf{V}})h_0(\lambda\xi\cdot x)\phi\|_{\mathcal{L}^{\infty}}\leq C(\left|\bar{n}\right|,\left|\bar{\mathbf{V}}\right|,\phi,h_0)\lambda^{-1}.
\end{align*}
It follows from $\bar{n}\cdot \xi_1=\bar{n}\cdot \xi_2=0$ that
\begin{align*}
	n''=-\lambda^{-6}\nabla\Delta^{-1}\Delta^{3}\nabla\cdot[\bar{n}h_6(\lambda\xi\cdot x)\phi]=-\lambda^{-6}\nabla\Delta^2[(\bar{n}\cdot \nabla \phi)h_6(\lambda\xi\cdot x)].
\end{align*}
Hence, we have $\text{supp }n''\subset \text{supp }\phi$, $\int n''dx=0$ and
\begin{align*}
	\|n''\|_{L^{\infty}}\leq C(\left|\bar{n}\right|,\phi,h_0)\lambda^{-1}.
\end{align*}
\quad For $\mathbf{V}''$, since $\bar{\mathbf{V}}\cdot \xi=0$, direct computations give
\begin{align*}
    \mathbf{V}''&=-\mathcal{R}[\nabla\cdot \mathbf{V}'-\mathbf{B}n]\\
    &=-\lambda^{-6}\mathcal{R}\{\Delta^3[\nabla\cdot(\bar{\mathbf{V}}h_6(\lambda\xi\cdot x)\phi)]-\mathbf{B}\Delta^3[\bar{n}h_6(\lambda\xi\cdot x)\phi]+\mathbf{B}\nabla\Delta^2[(\nabla\phi\cdot\bar{n})h_6(\lambda\xi\cdot x)]\}\\
    &=-\lambda^{-6}\mathcal{R}\{\Delta^3[\bar{\mathbf{V}}\cdot \nabla\phi h_6(\lambda \xi\cdot x)]+\mathbf{B}\Delta^2[\nabla((\nabla\phi\cdot\bar{n})h_6(\lambda\xi\cdot x))-\Delta(\bar{n}h_6(\lambda\xi\cdot x)\phi)]\}.
\end{align*}
Therefore, it can be written as
\begin{align*}
	\mathbf{V}''=-\lambda^{-6}\mathcal{R}[\Delta^2f],
\end{align*}
where
\begin{align*}
	f=\Delta[\bar{\mathbf{V}}\cdot \nabla\phi h_6(\lambda \xi\cdot x)]+[\nabla((\nabla\phi\cdot\bar{n})h_6(\lambda\xi\cdot x))-\mathbf{B}\Delta(\bar{n}h_6(\lambda\xi\cdot x)\phi)].
\end{align*}
Then, we have $\text{supp } \mathbf{V}''\subset \text{supp } f\subset \text{supp } \phi$, and $\int_{\mathbb{R}^n}\mathbf{V}''dx=0$. Furthermore, 
\begin{align*}
\|\mathbf{V}''\|_{L^{\infty}(\mathcal{O})}=\lambda^{-6}\|\mathcal{R}[\Delta^2f]\|_{L^{\infty}}\leq C(\left|\bar{n}\right|,\left|\bar{\mathbf{V}}\right|,\phi,h_0)\lambda^{-1}.
\end{align*}
Thus,
\begin{align*}
	\|\tilde{w}(x)-\bar{w}h_0(\lambda\xi\cdot x)\phi(x)\|_{L^{\infty}}\leq C(\bar{w},\phi,h_0)\lambda^{-1}.
\end{align*}

\quad For $x\in \mathcal{O}$, since $w_2-w_1\in \Lambda$, one has $\bar{w}(h_0\phi)(x)\in[w_1,w_2]$, which implies that
\begin{align*}
	\text{dist}(w+\tilde{w},[w_1,w_2])&\leq \text{dist}(w+\bar{w}h_0\phi,[w_1,w_2])+\left|\tilde{w}-\bar{w}h_0(\lambda\xi\cdot x)\phi\right|\\
	&\leq C(\bar{w},\phi,h_0)\lambda^{-1}.
\end{align*}
We define the disjoint open sets $\mathcal{O}_i$ as
\begin{align*}
	\mathcal{O}_i=\left\{x\in\mathcal{O}:\left|w+\bar{w}h_0(\lambda\xi\cdot x)\phi(x)-w_i\right|<\min\left(\dfrac{\varepsilon}{2},\dfrac{\left|w_2-w_1\right|}{4}\right)\right\}.
\end{align*}
Hence, for $x\in\mathcal{O}_i$, one has
\begin{align*}
	\left|w+\tilde{w}(x)-w_i\right|&\leq\left|w+\bar{w}h_0(\lambda\xi\cdot x)\phi-w_i\right|+\left|\tilde{w}-\bar{w}h_0(\lambda\xi\cdot x)\phi\right|\\
	&\leq\dfrac{\varepsilon}{2}+C(\bar{w},\phi,h_0)\lambda^{-1}.
\end{align*}
Clearly, if $\lambda$ is large enough, properties (ii) and (iii) holds.
\end{proof}

\quad An exact plane-wave solutions to \eqref{eq2} with compact support is identically zero. By Lemma \ref{lm2.6}, we can construct localized plane-wave solutions by introducing an arbitrarily small error in the range of the wave. In order to ensure the wave-cone is large enough for constructing localized plane-wave solutions, we need next lemma to commit the consistency.

\begin{lemma}\label{lm2.7} There exists a constant $C>0$ such that for any $(\bar{n},\bar{\mathbf{V}})\in \Lambda$ defined in \eqref{eq7}, there exists a sequence $(\tilde{n}_k,\tilde{\mathbf{V}}_k)\in C_c^{\infty}$ solving $\mathcal{L}(\tilde{n}_k,\tilde{\mathbf{V}}_k)=0$ and satisfies:
\begin{enumerate}[(1)]
	\item [(i)] {$dist((\tilde{n}_k,\tilde{\mathbf{V}}_k),[-(\bar{n},\bar{\mathbf{V}}),(\bar{n},\bar{\mathbf{V}})])\rightarrow 0$ uniformly in $Q_1$ as $k\rightarrow \infty$;}
	\item [(ii)] {$(\tilde{n}_k,\tilde{\mathbf{V}}_k)\rightarrow 0$ as $k\rightarrow \infty$ in the sense of distribution;}
	\item [(iii)] {$\int_{Q_1}\left|(\tilde{n}_k,\tilde{\mathbf{V}}_k)\right|^2(x)dx\geq C\left|(\bar{n},\bar{\mathbf{V}})\right|^2$ for all $k\in \mathbb{N}$.}
\end{enumerate}
\end{lemma}

\begin{proof}
Let $\bar{w}=(\bar{n},\bar{\mathbf{V}})\in\Lambda$ and consider the segment $[-\bar{w},\bar{w}]$. For $k\in \mathbb{N}$, set $w=0=\frac{1}{2}(-\bar{w})+\frac{1}{2}\bar{w}$, $\varepsilon_k=\frac{2^{-nk}}{k}$, and $\mathcal{O}=Q_{2^{-k}}$. Then by Lemma \ref{lm2.6}, there exists $v_k\in C_c^{\infty}(Q_{2^{-k}})$ satisfying $\int v_k(x)dx=0$, $\mathcal{L}v_k=0$, and
\begin{align*}
	\text{dist}(v_k,[-\bar{w},\bar{w}])<\dfrac{2^{-nk}}{k}.
\end{align*}
Furthermore, there exist $\mathcal{O}_i^{(k)}$ such that
\begin{align*}
	\left|v_i(x)-(-1)^i\bar{w}\right|<\dfrac{2^{-nk}}{k}\ \text{for}\ x\in\mathcal{O}_i^{(k)}\quad\text{and}\quad\left|\mathcal{H}^n(\mathcal{O}_i^{(k)})-\dfrac{1}{2^{nk}+1}\right|<\dfrac{2^{-nk}}{k}.
\end{align*}
\quad Next, consider the set of disjoint cubes $\{Q_{2^{-k}}(x^{(j,k)})\}_{j=1}^{2^{nk}}$ which satisfying $Q_1=\cup Q_{2^{-k}}(x^{(j,k)})$. Denote
\begin{align*}
	\tilde{w}_k=\mathop{\Sigma}\limits_{j=1}\limits^{2^{nk}}a_j^{(k)}\quad\text{where}\quad a_j^{(k)}(x)=v_k(x-x^{(j,k)})\in\mathcal{C}_c^{\infty}(Q_{2^{-k}}(x^{(j,k)})).
\end{align*}
Then, $\mathcal{L}\tilde{w}_k=0$, and $\text{dist}(\tilde{w}_k,[-\bar{w},\bar{w}])\rightarrow 0$ as $k\rightarrow \infty$.

\quad Since $a_j^{(k)}(x)\in \mathcal{L}(\mathbb{R}^n)$ and $\text{diam}(\text{supp}\  a_j^{(k)})\leq \sqrt{n}2^{-k}\rightarrow 0$, then $\tilde{w}_k\rightarrow 0$. Note that
\begin{align*}
	\int_{Q_1}\left|\tilde{w}_k\right|^2=\mathop{\Sigma}\limits_{j=1}\limits^{2^{nk}}\int\left|a_j^{(k)}\right|^2=2^{nk}\int_{2^{-k}}\left|v_k\right|^2dx.
\end{align*}
Therefore, one has
\begin{align*}
	2^{nk}\int_{Q_{2^{-k}}}\left|v_k\right|^2dx-\left|\bar{w}\right|^2&=\dfrac{1}{\left|Q_{2^{-k}}\right|}\int_{Q_{2^{-k}}}\left|v_k\right|^2-\left|\bar{w}\right|^2dx\\
	&=\dfrac{1}{\left|Q_{2^{-k}}\right|}\left(\mathop{\Sigma}\limits_{i=1}\limits^{2}\int_{\mathcal{O}_i}\left|v_k\right|^2-\left|\bar{w}\right|^2dx+\int_{Q_{2^{-k}}\backslash(\mathcal{O}_1\cup\mathcal{O}_2)}\left|v_k\right|^2-\left|\bar{w}\right|^2dx\right)\\
	&\leq C\dfrac{1}{k}.
\end{align*}
Hence, for sufficiently large $k$, one has
\begin{align*}
	\int_{Q_1}\left|\tilde{w_k}\right|^2dx\geq\dfrac{1}{2}\left|\bar{w}\right|^2.
\end{align*}
This finishes the proof of the lemma.
\end{proof}

\section{Convex Integration}
\quad In this section, we will use the convex integration method to construct weak solutions, where the iteration scheme follows \cite{MR3459023} and \cite{MR3505175}.

\begin{lemma}\label{lm3.1} Let $n\geq 3$, for every $\bar{w}=(\bar{m},\bar{\mathbf{U}})\in\mathcal{U}_{\rho,q}:=\mathring{K}_{\rho,q}^{\text{co}}$, there exists a subsolution $w=(m,\mathbf{U})\in C_{c}^{\infty}$ such that
	\begin{enumerate}[(1)]
		\item [(i)] {$\bar{w}+w(x)\in\mathcal{\mathbf{U}}_{\rho,q}$ for all $x\in Q_1$;}
		\item [(ii)] {$\int_{Q_1}\left|w(x)\right|^2dx\geq\Phi(\text{dist}(\bar{w},K_{\rho,q}))$.}
	\end{enumerate}
\end{lemma}

\begin{proof}
	Let $\bar{w}\in\mathcal{\mathbf{U}}_{\rho,q}=\mathring{K}_{\rho,q}^{\text{co}}$. It follows from Lemmas \ref{lm2.3} and \ref{lm2.4}, that there exists a $\tilde{w}\in\Lambda$ satisfying
	\begin{align*}
		[(\bar{m}-\tilde{m},\bar{\mathbf{U}}-\tilde{\mathbf{U}}),(\bar{m}+\tilde{m},\bar{\mathbf{U}}+\tilde{\mathbf{U}})]\in\mathring{K}_{\rho,q}^{\text{co}},\quad\text{dist}(\sigma_{\bar{m}},\partial K_{\rho,q}^{\text{co}})\geq\frac{1}{2}\text{dist}(\bar{w},\partial K_{\rho,q}^{\text{co}}),
	\end{align*}
	and 
	\begin{align*}
		\left|\tilde{m}\right|\geq\frac{C}{\sqrt{n\rho q}}(n\rho q-\left|\bar{m}\right|^2).
	\end{align*}
	By Lemma \ref{lm2.6}, for $\varepsilon<\frac{1}{4}\text{dist}(\bar{w},\partial K_{\rho,q}^\text{co})$, we can find a subsolution $w=(m,\mathbf{U})\in C_c^{\infty}$ such that $\bar{w}+w\in\mathcal{\mathbf{U}}_{\rho,q}$ for all $x\in Q_1$ and
	\begin{align*}
		\int_{Q_1}\left|m(x)\right|^2dx\geq\frac{1}{2}\left|\tilde{m}\right|^2\geq\frac{C}{\sqrt{n\rho q}}(n\rho q-\left|\bar{m}\right|^2)^2.
	\end{align*}
\end{proof}

\quad Lemma \ref{lm3.1} allows us to find suitable line segment for certain constrain set. Since the constrain sets rely on $x$, we need the following lemma, which gives more accurate estimates.

\begin{lemma}\label{lm3.2}
	Suppose that $\tilde{K}\subset K_{\bar{\rho},\bar{q}}$, and $\bar{w}\in \mathring{\tilde{K}}^{co}$ is a constant vector. Given a bounded open set $\mathcal{\mathbf{U}}$ and $\varepsilon>0$, there exists a $\tilde{w} \in C_c^{\infty}(\mathcal{\mathbf{U}})$ such that the following statements holds:
	\begin{enumerate}[(1)]
		\item [(i)] {$\int \tilde{w}(x)dx=0\ \text{and}\ \mathcal{L}\tilde{w}=0$;}
		\item [(ii)] {there is a constant $\gamma=\gamma(\varepsilon,\bar{w}.\tilde{K})$ such that $Q_{\gamma}+\bar{w}+\tilde{w}(x)\in \mathring{\tilde{K}}^{co}$;}
		\item [(iii)] {$\frac{1}{\left|\mathcal{\mathbf{U}}\right|}\int_{\mathcal{\mathbf{U}}}dist(\bar{w}+\tilde{w}(x),\tilde{K})dx<\varepsilon$.}
	\end{enumerate}
\end{lemma}

\begin{proof}
	For $\bar{w}\in \mathring{\tilde{K}}^{co}$, there exists a finite set $\{w_i\}_{i=1}^{N}\subset \tilde{K}$ such that $\bar{w}\in \text{int}(\{w_i\}_{i=1}^{N})^{\text{co}}$. For $\beta\in(0,1)$, define
	\begin{align*}
		L_{\beta}=\{w_i^{(\beta)}=\bar{w}+(1-\beta)(w_i-\bar{w})\}_{i=1}^{N}.
	\end{align*}
	Then $\bar{w}\in \mathring{L}_{\beta}^{\text{co}}$ for any $\beta\in(0,1)$. For $0<\beta_2<\beta_1<1$, one has
	\begin{align*}
		\overline{L_{\beta_1}^{\text{co}}}\subset \mathring{L}_{\beta_2}^{\text{co}}\subset \mathring{\tilde{K}}^{\text{co}}.
	\end{align*}
	
	We choose $\delta$ such that $\delta\underset{i}{\max}\left|w_i-\bar{w}\right|<\frac{\varepsilon}{4}$.
	
	\quad Let $L^{(0)}=L_{\delta}$, and
	\begin{align*}
		L^{j+1}:=L^{(j)}\cup\left\{\upsilon_1w_1'+\upsilon_2w_2':w_i'\in L^{(j)},w_2'-w_1'\in\Lambda,\upsilon_i\in(0,1),\upsilon_1+\upsilon_2=1\right\}.
	\end{align*}
	Then, for $j\geq 0$, we have $L^{(j)}\subset L_{\delta}^{\text{co}}$. We claim further that, for $j\geq 0$,
	\begin{align}\label{eq8}
		\left\{\mathop{\Sigma}\limits_{i=1}\limits^{j+1}\upsilon_iw_i'|w_i'\in L^{(0)}	,\upsilon_i\geq 0,\mathop{\Sigma}\limits_{i=1}\limits^{j+1}\upsilon_i=1\right\}\subset L^{(j)}.
	\end{align}
    \quad First, \eqref{eq8} is trivially true for $j=0$. Suppose that \eqref{eq8} holds for $0\leq j\leq k$. For $j=k+1$, let $w=\mathop{\Sigma}\limits_{i=1}\limits^{k+2}\upsilon_iw_i'$ for $w_i'\in L^{(0)}$, $\upsilon_i>0$, and $\mathop{\Sigma}\limits_{i=1}\limits^{k+2}\upsilon_i=1$, we can rewrite $w$ as
	\begin{align*}
		w=\dfrac{\upsilon_1}{\upsilon_1+\upsilon_2}w'+\dfrac{\upsilon_2}{\upsilon_1+\upsilon_2}w'',
	\end{align*}
	where
	\begin{align*}
		w'=(\upsilon_1+\upsilon_2)w_1'+\mathop{\Sigma}\limits_{i=3}\limits^{k+2}\upsilon_iw_i'\ \text{and}\ w''=(\upsilon_1+\upsilon_2)w_2'+\mathop{\Sigma}\limits_{i=3}\limits^{k+2}\upsilon_iw_i'.
	\end{align*}
	Then, $w',w''\in L^{(k)}$, and one has
	\begin{align*}
		w'-w''=(\upsilon_1+\upsilon_2)(w_1'-w_2').
	\end{align*}
	Since $w_i'\in L^{(0)}$, then there exists $w_{j_i}\in \tilde{K}$ such that $w_i'=w_{j_i}^{(\delta)}$. Hence
	\begin{align*}
		w_1'-w_2'=w_{j_1}^{(\delta)}-w_{j_2}^{(\delta)}=(1-\delta)(w_{j_1}-w_{j_2}).
	\end{align*}
	Since $w_{j_1},w_{j_2}\in \tilde{K}$, then $w_{j_1}-w_{j_2}\in \Lambda$. Hence, $w'-w''\in \Lambda$, that is, $w\in L^{(k+1)}$. Then the claim \eqref{eq8} follows.
	
	\quad Then, we conclude that $L^{(N_n)}=L_{\delta}^{co}$, where $N_n=\dfrac{n(n-1)}{2}-1$. Let $\tau_j=2^{-j}$ for $j=1,2,...,N_n$. We claim that for any $w\in L^{(j)}$ and open set $\mathcal{O}\subset \mathcal{U}$, there exists a $\tilde{w}\in \mathcal{C}_c^{\infty}(\mathcal{O})$ satisfying $\mathcal{L}\tilde{w}=0$ and
	\begin{align*}
		w+\tilde{w}\subset \mathring{L}_{\tau_j\delta}^{\text{co}},\quad \dfrac{1}{\left|\mathcal{O}\right|}\int_{\mathcal{O}}\text{dist}(w+\tilde{w}(x),\tilde{K})dx<(1-\tau_j)\varepsilon.
	\end{align*}
	We prove this claim by induction. For $j=0$, if $w\in L^{(0)}=L_{\delta}$, then $w=w_i^{(\delta)}$ for some $i$. Then, $w\in L_{\delta}\subset \mathring{L}_{\tau_0\delta}^{co}$ and $\mathcal{L}\tilde{w}=0$ holds for $\tilde{w}=0$. Furthermore, we have
	\begin{align*}
		\text{dist}(w,\tilde{K})\leq\left|w_i-w_i^{\delta}\right|\leq\underset{j}{\max}\left|w_j-w_j^{\delta}\right|\leq\delta\underset{j}{\max}\left|w_j-\bar{w}\right|<\dfrac{\varepsilon}{4}<(1-\tau_0)\varepsilon.
	\end{align*}
	\quad Assume that the claim holds for $0\leq j\leq k$. For the case $j=k+1$, if $w\in L^{(k+1)}\backslash L^{(k)}$, then there exist $w_1',w_2'\in L^{(k)}$ such that
	\begin{align*}
		w=\upsilon_1w_1'+\upsilon_2w_2',\quad w_2'-w_1'\in\Lambda,\quad \upsilon_i\in(0,1),\quad \upsilon_1+\upsilon_2=1.
	\end{align*}
	Then, for a given open subset $\mathcal{O}\subset\mathcal{U}$ and $\varepsilon_0$, by Lemma \ref{lm2.6}, there exists a $\tilde{w}_0\in C_c^{\infty}(\mathcal{O})$ satisfying $\mathcal{L}\tilde{w}=0$ and
	\begin{align*}
		\text{dist}(w+\tilde{w}_0,[w_1',w_2'])<\varepsilon_0,
	\end{align*}
	and exist two disjoint open subsets $\mathcal{O}_1,\mathcal{O}_2\subset \mathcal{O}$ satisfying
	\begin{align*}
		\left|w+\tilde{w}_0-w_i'\right|<\upsilon_0\quad \text{in}\quad  \mathcal{O}_i,\quad \left|\mathcal{H}^n(\mathcal{O}_i)-\upsilon_i\mathcal{H}^n(\mathcal{O})\right|<\varepsilon_0\quad \text{for}\quad i=1,2.
	\end{align*}
    Let $\tilde{w}=\tilde{w}_0+\tilde{w}_1+\tilde{w}_2$, then $\tilde{w}\in \mathcal{C}_c^{\infty}$ and satisfies $\mathcal{L}\tilde{w}=0$. Furthermore, for $x\in \mathcal{O}\backslash(\mathcal{O}_1\cup\mathcal{O}_2)$, one has
	\begin{align*}
		w+\tilde{w}(x)=w+w\tilde{w}_0(x)\in Q_{\varepsilon_0}+[w_1',w_2']\subset Q_{\varepsilon_0}+L_{\tau_k\delta}^{\text{co}}.
	\end{align*}
	\quad For $x\in\mathcal{O}_i$($i=1,2$), it holds that
	\begin{align*}
		w+\tilde{w}(x)=(w+\tilde{w}_0(x)-w_i')+(w_i'+\tilde{w}_i(x))\in Q_{\varepsilon_0}+L_{\tau_k\delta}^{\text{co}}.
	\end{align*}
	Since $\overline{L_{\tau_k\delta}^{\text{co}}}\subset \mathring{L}_{\tau_{k+1}\delta}^{\text{co}}$, then for $\epsilon$ sufficiently small, we have
	\begin{align*}
		Q_{2\varepsilon_0}+L_{\tau_k\delta}^{\text{co}}\subset \mathring{L}_{\tau_{k+1}\delta}^{\text{co}}.
	\end{align*}
	Then, $w+\tilde{w}\subset \mathring{L}_{\tau_{k+1}\delta}^{\text{co}}$.
	
	\quad For $x\in \mathcal{O}_i$, $w+\tilde{w}(x)=w+\tilde{w}_0(x)+\tilde{w}_i(x)$. Directly computations show that
	\begin{align*}
		dist(w+\tilde{w}(x),\tilde{K})&\leq\left|w+\tilde{w}_0(x)-w_i'\right|+dist(w_i'+\tilde{w}_i(x),\tilde{K})\\
		&\leq\varepsilon_0+\text{dist}(w_i'+\tilde{w}_i'(x),\tilde{K}).
	\end{align*}
	Since $\mathcal{O}_1$ and $\mathcal{O}_2$ are disjoint open subsets of $\mathcal{O}$, one has
	\begin{align*}
		\mathcal{H}^{n}(\mathcal{O}\backslash(\mathcal{O}_1\cup\mathcal{O}_2))&=\mathcal{H}^{n}(\mathcal{O})-\mathop{\Sigma}\limits_{i=1}\limits^{2}\mathcal{H}^{n}(\mathcal{O}_i)\\
		&\leq\mathop{\Sigma}\limits_{i=1}\limits^{2}\left|\upsilon_i\mathcal{H}^{n}(\mathcal{O})-\mathcal{H}^{n}(\mathcal{O}_i)\right|<2\varepsilon_0.
	\end{align*}
	Note that $\tilde K$ is bounded, $\text{dist}(w+\tilde{w}(x),\tilde{K})\leq 2M(\bar{\rho},\bar{q})<\infty$, it follows that
	\begin{align*}
		\int_{\mathcal{O}}\text{dist}(w+\tilde{w}(x),\tilde{K})dx&\leq\mathop{\Sigma}\limits_{i=1}\limits^{2}\int_{\mathcal{O}_i}\text{dist}(w+\tilde{w}(x),\tilde{K})dx+2M(\bar{\rho},\bar{q})\left|\mathcal{O}\backslash(\mathcal{O}_1\cup\mathcal{O_2})\right|\\
		&\leq\mathop{\Sigma}\limits_{i=1}\limits^{2}\int_{\mathcal{O}_i}\int\varepsilon_0+\text{dist}(w_i'+\tilde{w}_i(x),\tilde{K})dx+4M(\bar{\rho},\bar{q})\varepsilon_0\\
		&\leq(\left|\mathcal{O}_1\right|+\left|\mathcal{O}_2\right|)(1-\tau_k)\varepsilon+C\varepsilon_0.
	\end{align*}
    Hence
	\begin{align*}
		\dfrac{1}{\left|\mathcal{O}\right|}\int_{\mathcal{O}}\text{dist}(w+\tilde{w}(x),\tilde{K})dx\leq(1-\tau_k)\varepsilon+\dfrac{C\epsilon_0}{\left|\mathcal{O}\right|}<(1-\tau_{k+1})\varepsilon,
	\end{align*}
	where $\tau_{k+1}<\tau_k$ has been used, which proves the claim. 
	
	\quad Therefore, for $\bar{w}\in L_{\delta}^{\text{co}}$, there exists a $\tilde{w}\in \mathcal{C}_c^{\infty}(\mathcal{U})$ satisfying $\mathcal{L}\tilde{w}=0$ and
	\begin{align*}
		\bar{w}+\tilde{w}\subset \mathring{L}_{\tau_{N_n}\delta}^{\text{co}},\quad \dfrac{1}{\left|\mathcal{U}\right|}\int_{\mathcal{U}}\text{dist}(\bar{w}+\tilde{w}(x),\tilde{K})dx<(1-\tau_{N_n})\varepsilon.
	\end{align*}
	Hence, one can choose $\gamma>0$ sufficiently small such that
	\begin{align*}
		Q_{\gamma}+\bar{w}+\tilde{w}\subset Q_{\gamma}+\text{int} L_{\tau_{N_n}\delta}^{\text{co}}\subset \text{int} \tilde{K}^{\text{co}}.
	\end{align*}
	This finishes the proof of the lemma.
\end{proof}

\quad As mentioned above, our definition to the strict subsolutions demand the map $x\mapsto K_x$ is continuous in a open set $\mathcal{D}$, the following two lemmas ensure the inclusion relation fix along the variation of $x$ and will help us deal with the variation of the constrain set $K_x$.

\begin{lemma}\label{lm3.3}
    Assume that $K_x\subset \mathbb{R}^n$, $x\mapsto K_x$ is continuous in an open set $\mathcal{D}$ in the Hausdorff distance. Then for any compact set $C\subset\mathring{K}_x^{\text{co}}$, there exists $\varepsilon>0$ such that for any $x'\subset \mathcal{D}$ with $d_{\mathcal{H}}(x,x')<\varepsilon$,
    \begin{align*}
        C\subset\mathring{K}_{x'}^{\text{co}}.
    \end{align*}
\end{lemma}
\begin{proof}
    Since $\mathring{K}_{x}^{\text{co}}$ is open, then for any $x\in C$, there exists an simplex $S_x$ and corresponding $\{v_1,v_2,\cdots,v_{n+1}\}$ such that
    \begin{align*}
        x\in I_x:=\left\{\mathop{\sum}\limits_{i=1}^{n+1}\lambda_iv_i:\lambda_i\in\left(\frac{1}{2(n+1)},\frac{2}{n+1}\right),\mathop{\sum}\limits_{i=1}^{n+1}\lambda_i=1,i=1,\cdots,n+1\right\}.
    \end{align*}
    \quad Since $C$ is compact, and the simplex $I_x$ covers $C$, one can choose a sub-cover $\{I_{x_k}\}_{k=1,\cdots m}$ of $C$. For fixed $k=1,\cdots,m$, denote the simplex as $S:=S_{x_k}$, and the denote the corresponding points as $\{v_1,\cdots,v_{n+1}\}\subset\mathring{K}_x^{\text{co}}$. Let $I:=I_{x_k}$. If $\varepsilon<\text{dist}(\partial I,\partial S)$, then for any $v_i'\in B_{\varepsilon}(v_i),\ i=1,\cdots,n+1$, one has
    \begin{align*}
        I\subset\{v_1',v_2',\cdots,v_{n+1}'\}^{\text{co}}.
    \end{align*}
    Then for any $\varepsilon>0$ and $i=1,\cdots,n+1$, $B_{\varepsilon}(v_i)$ contains some point $v_i'\in K_{x'}$. In fact, by Carath\'{e}odory theorem, $v_i$ can be expressed as $\mathop{\sum}\limits_{j=1}^{n+1}\mu_jz_j$, where $z_j\in K_x,\ \mu_j\in[0,1]$ and $\mathop{\sum}\limits_{j=1}^{n+1}\mu_j=1$. Since $x\mapsto K_x$ is continuous in the Hausdorff distance, then for any $x'\subset\mathcal{D}$ with $d_{\mathcal{H}}(x,x')<\varepsilon$, there exist $z_j'\in K_{x'},\ j=1\cdots n+1$ such that $z_j'\in B_{\varepsilon}(z_j)$. Let
    \begin{align*}
        v_i'=\mathop{\sum}\limits_{j=1}^{n+1}\mu_jz_j'.
    \end{align*}
    Direct computations yield
    \begin{align*}
        \left|v_i-v_i'\right|\leq\mathop{\sum}\limits_{j=1}^{n+1}\mu_j\left|z_j-z_j'\right|<\varepsilon'.
    \end{align*}
   Hence, for sufficiently small $\varepsilon$, $I\subset\{v_1',v_2',\cdots,v_{n+1}'\}^{\text{co}}$. Since $v_i'\in K_{x'}^{\text{co}}$, then $I\subset K_{x'}^{\text{co}}$. Since $I$ is open, then $I\subset\mathring{K}_{x'}^{\text{co}}$. Since there are infinitely many such simplices, we can choose $\varepsilon>0$ such that the inclusion relation $I_{x_k}\subset \mathring{K}_{x'}^{\text{co}}$ holds for any $k=1,\cdots,m$. Hence,
    \begin{align*}
        C\subset\mathop{\cup}\limits_{k=1,\cdots,m}I_{x_k}\subset \mathring{K}_{x'}^{\text{co}}.
    \end{align*}
    This finishes the proof of the lemma.
\end{proof}

\begin{lemma}\label{lm3.4}
	Suppose that $(\rho,w,q)$ is a strict subsolution in a bounded open set $\mathcal{D}$ with constrain sets $K_x$. For given $\varepsilon>0$, there exists a compact set $\mathcal{C}\subset{D}$ and a sequence $\{w_k\}$ such that $\{\rho,w_k,q\}$ are strict subsolutions in $D$ and $w_k-w\in C_c^{\infty}$, $\text{supp}(w_k-w)\subset\mathcal{C}$, $w_k\rightarrow w$.
	
	\quad Furthermore, it holds that
	\begin{align*}
		\int_{\mathcal{D}}\text{dist}(w_k(x),K_x)dx\leq \varepsilon.
    \end{align*}
\end{lemma}
\begin{proof}
	Set $M=2\underset{x\in\mathcal{D}}{\sup}\underset{w\in K_x}{\sup}\left|w\right|$. Let $\mathcal{C}'$ be a compact subset such that $\mathcal{H}^{n}(\mathcal{D}\backslash\mathcal{C}')<\dfrac{\varepsilon}{4M}$. We now choose any point $\zeta\in\mathcal{D}$, and open bounded set $\mathcal{U}$, set $\varepsilon=\dfrac{\varepsilon}{4\left|\mathcal{D}\right|}$, then by Lemma \ref{lm3.2}, there exists a $v\in \mathcal{C}_c^{\infty}(\mathcal{U})$ satisfying $\mathcal{L}w=0$ and
	\begin{align*}
		Q_{\gamma}+w(\zeta)+v\subset \mathring{K}_{\zeta}^{\text{co}},\quad \dfrac{1}{\left|\mathcal{U}\right|}\int_{\mathcal{U}}dist(w(\zeta)+v(x),K_{\zeta})dx<\dfrac{\varepsilon}{4\left|\mathcal{D}\right|}.
	\end{align*}
	Denote $\mathfrak{K}_{\zeta,\gamma}=\{\bar{w}\in \mathring{K}_{\zeta}^{co}:\text{dist}(\bar{w},\partial K_{\zeta}^{\text{co}})\geq \dfrac{\gamma}{2}\}$. It holds that
	\begin{align*}
		Q_{\gamma/2}+w(\zeta)+v\subset\overline{\mathfrak{K}_{\zeta,\gamma}}\subset \mathring{K}_{\zeta}^{\text{co}}.
	\end{align*}
	It follows from Lemma \ref{lm3.2} that there exists a positive number $r(\zeta)>0$ such that
	\begin{align*}
		w(x)+v\subset{\mathfrak{K}_{\zeta,\gamma}}\subset \mathring{K}_x^{\text{co}}\ for\ x\in Q_{r(\zeta)}(\zeta).
	\end{align*}
	\quad Let $\{\mathcal{O}^{i}=Q_{r(\zeta^{i})}(\zeta^{i})\}_{i=1}^{N}$ be a finite covering of $\mathcal{C}'$ such that $\mathop{\cup}\limits_{i=1}^{N}\bar{\mathcal{O}^i}\subset \mathcal{D}$. Set $\mathcal{C}:=\mathop{\cup}\limits_{i=1}^{N}\bar{\mathcal{O}^i}$ and $r_0=\dfrac{1}{2}\underset{i}{\min}r(\zeta^{i})$. Then by Whitney covering lemma, for the open sets $\mathcal{O}^1$ and $\{\mathcal{O}^i\backslash(\mathop{\cup}\limits_{j=1}^{i-1}\bar{\mathcal{O}^{j}})\}_{i=2}^{N}$, there exists disjoint open cubes $\{\tilde{Q}^{l}\}_{l=1}^{\infty}$ with $\mathcal{H}^{n}(\mathcal{C}\backslash(\cup_l\tilde{Q}^l))=0$, each $\tilde{Q}^{l}$ lying in some $\mathcal{O}^{i(l)}$. For $k=1,2,...$, decomposing the cubes $\tilde{Q}^{l}$, then there exist finitely many disjoint open cubes $\{Q^{(j,k)}\}=\{Q_{r_jk}(x^{j,k})\}_{j=1}^{J_k}$ satisfying
	\begin{align*}
		\underset{1\leq j\leq J_k}{\max}r_{jk}<2^{-k}r_0,\ \mathcal{H}^{n}(\mathcal{C}\backslash(\mathop{\cup}\limits_{j=1}^{J_k}Q^{j,k}))<\dfrac{\epsilon}{4C_0}.
	\end{align*}
	and for each $j=1,2,...,J_k$, $Q^{j,k}\subset \mathcal{O}^{i(j,k)}$ for some $i(j,k)$.
	
	\quad Then, apply Lemma \ref{lm3.2} to $w(\zeta^{i(j,k)})$ with $\mathcal{U}=Q^{j,k}$ and $\varepsilon=\dfrac{\varepsilon}{4\left|\mathcal{D}\right|}$, one gets $v_{jk}(x)\in\mathcal{C}_c^{\infty}(Q^{(j,k)})$ which satisfies $\mathcal{L}w=0$,
	\begin{align*}
		Q_{\gamma}+w(\zeta^{i(j,k)})+v_{jk}\subset\mathring{K}_{\zeta^{i(j,k)}}^{\text{co}},
	\end{align*}
	and
	\begin{align*}
		\dfrac{1}{\left|Q^{j,k}\right|}\int_{Q^{(j,k)}}\text{dist}(w(\zeta^{i(j,k)})+v_{jk}(x),K_{\zeta^{i(j,k)}})dx<\dfrac{\varepsilon}{4\left|\mathcal{D}\right|}.
	\end{align*}
	\quad Define
	\begin{align*}
		\tilde{w}_k(x)=\mathop{\Sigma}\limits_{j=1}\limits^{J_k}v_{jk}(x)\quad \text{and}\quad w_k(x)=w(x)+\tilde{w}_k(x).
	\end{align*}
	Then, $\tilde{w}_k\in\mathcal{C}_c^{\infty}(\mathcal{D})$ and $\text{supp}\tilde{w}_k\subset\mathcal{C}$. Since $\mathcal{L}v_{jk}=0$, then $(\rho,w_k,q)$ solves \eqref{eq2}. Moreover, since $Q^{j,k}\subset \mathcal{O}^{i(j,k)}$, for $x\in Q^{j,k}$,
	\begin{align*}
		w_k(x)=w(x)+v_{jk}(x)\subset \mathring{K}_x^{\text{co}}.
	\end{align*}
	\quad For $x\in \mathcal{D}\backslash(\cup_jQ^{j,k})$, $w_k(x)=w(x)\in \mathring{K}_x^{\text{co}}$. Therefore, $(\rho,w_k,q)$ is a strict subsolution in $\mathcal{D}$ with the constrain set $K_{x}$.
	
	\quad Furthermore, by direct computations, one has
	\begin{align*}
		\int_{Q^{(j,k)}}\text{dist}(w_k(x),K_x)dx&=\int_{Q^{(j,k)}}\text{dist}(w(x)+v_{jk}(x),K_x)dx\\
		&\leq\int_{Q^{(j,k)}}\text{dist}(w(\zeta^{i(j,k)})+v_{jk}(x),K_{\zeta^{i(j,k)}})+\dfrac{\varepsilon}{4\left|\mathcal{D}\right|}dx\\
		&\leq\dfrac{\left|Q^{j,k}\right|}{2\left|\mathcal{D}\right|}\varepsilon.
	\end{align*}
	Hence,
	\begin{align*}
		\int_{D}\text{dist}(w_k(x),K_x)dx&=\mathop{\Sigma}\limits_{j=1}\limits^{J_k}\int_{Q^{(j,k)}}\text{dist}(w_k(x),K_x)dx+\int_{\mathcal{D}\backslash(\cup_jQ^{j,k})}\text{dist}(w_k(x),K_x)dx\\
		&\leq\left|\mathcal{D}\right|\dfrac{\epsilon}{2\left|\mathcal{D}\right|}+C_0\left|\mathcal{D}\backslash(\cup_jQ^{j,k})\right|\\
		&\leq\dfrac{\varepsilon}{2}+C_0\left|\mathcal{D}\backslash\mathcal{C}'\right|\leq \epsilon.
	\end{align*}
	Therefore, for $\phi\in\mathcal{C}_c^{\infty}(\mathbb{R}^n)$, one has
	\begin{align*}
		\left|\int_{\mathbb{R}^n}\tilde{w}_k\phi dx\right|=\left|\mathop{\Sigma}\limits_{j=1}^{J_k}\int_{Q^{(j,k)}}v_{jk}\phi dx\right|.
	\end{align*}
    This finishes the proof of the lemma.
\end{proof}

\begin{proposition}\label{pp3.1}
    If $(\rho,\bar{m},\bar{\mathbf{U}},q)$ is a strict subsolution in $\mathcal{D}$ with constrain sets $K_x$, then there exist infinitely many pairs $w=(m,\mathbf{U})$ such that $(\rho,m,\mathbf{U},q)$ are subsolutions and
	\begin{align*}
		(m,\mathbf{U})(x)\in K_x\quad \text{for a.e.}\quad x\in \bar{D},\quad \text{and}\quad \text{supp}(m-\bar{m},\mathbf{U}-\bar{\mathbf{U}})\subset\bar{D}.
	\end{align*}
\end{proposition}

\begin{proof}
	Define
	\begin{align*}
		X_0=\left\{w\in L^{\infty}_{w\ast}:(\rho,w,q)\ \text{are strict subsolutions in } \mathcal{D}\ \text{with}\ K_x\ \text{and}\ w-\bar{w}\in C_c^{\infty}\right\}.
	\end{align*}
	
	\quad Since $\bar{w}\in X_0$, the set $X_0$ is non-empty. Since $X_0$ is a bounded set, we denote $X$ be the closure of $X_0$ in $C_{\text{loc}}$ topology. By Banach-Alaoglu theorem, $X$ is metrizable, hence there exists a metric $d$ such that $(X,d)$ is a complete metric space. Hence, it suffices to prove that the set of $w=(m,\mathbf{U})$ is an infinite dense set in $X$.

	\quad For any $w'\in X_0$, and $\varepsilon_0>0$, we aim to construct a sequence $\{w_k\}\subset X_0$ which converge to $w\in X$, such that $d(w_k,w')\leq \varepsilon_0$, and for any compact subset $\mathcal{C}\subset\mathcal{D}$, we have
	\begin{align*}
		w_k\rightarrow w\quad \text{strongly in }L^2\quad \text{and}\quad \int_C \text{dist}(w_k(x),K_x)dx\rightarrow 0.
	\end{align*}
	
	If such $\{w_k\}$ exists, then $w=\lim w_k$, and $w$ satisfies $d(w',w)\leq \varepsilon_0$. Therefore, $w=(m,\mathbf{U})$ is a dense set in $X_0$, and thus $w$ is a dense set in $X$. Since $w\notin X_0$, $\{w_k\}$ is an infinite sequence in $X_0$, hence $X$ is an infinite set.
	
	\quad Let $\{\mathcal{D}_j\}_{j=1}^{\infty}$ be a family of bounded open subsets of $\mathcal{D}$ such that $\mathcal{D}_1\subset\mathcal{D}_2\subset...$, and $D=\cup\mathcal{D}_j$. Let $w_1=w'$. Given $w_k=(m_k,\mathbf{U}_k)\in X_0$, we construct the sequence $w_{k+1}$ as follows.
	
	\quad Assume $(\rho,m_k,\mathbf{U}_k,q)$ is a strict subsolution in $\mathcal{D}_{k+1}\subset\mathcal{D}$ with constrain set $K_{\rho,q}$, by Lemma \ref{lm3.4}, there exists a sequence $\{w^{(i,k)}\}_{i=1}^{\infty}$ such that $(\rho,w^{(i,k)},q)$ are strict subsolutions in $\mathcal{D}_{k+1}$ with constrain set $K_{\rho,q}$, and satisfies
	\begin{align*}
		w^{i,k}-w_k\in C_c^{\infty}(\mathcal{D}_{k+1}),\quad w^{i,k}\rightarrow w_k,
	\end{align*}
	and
	\begin{align*}
		\int_{\mathcal{D}_{k+1}}\text{dist}(w^{i,k}(x),K_x)dx\leq\{2^{-k},\frac{1}{2}\int_{\mathcal{D}_1}\text{dist}(w_k(x),K_x)dx\}.
	\end{align*}
	\quad Since $w^{(i,k)}(x)\in \mathring{K}_{\rho(x),q(x)}^{\text{co}}$ for $x\in \mathcal{D}_{k+1}$ and  $\text{supp}(w^{i,k}-w_k)\subset \mathcal{D}_{k+1}$, it follows that $w^{i,k}(x)\in \mathring{K}_{\rho(x),q(x)}^{\text{co}}$ for $x\in\mathcal{D}$. Then, $(\rho,w^{(i,k)},q)$ are strict subsolutions in $\mathcal{D}$. Since $w^{(i,k)}-\bar{w}=(w^{(i,k)}-w_k)+(w_k-w)\in C_c^{\infty}$, then $w^{(i,k)}\in X_0$. Furthermore, since $w^{(i,k)}\rightarrow w_k$, then there exists a sufficiently large $i_k$ such that
	\begin{align*}
		\max_{1\leq j<k}\left|\int_{\mathcal{D}_j}(w^{(i_k,k)}-w_k)w_kdx\right|<\min\left\{2^{-k},\frac{1}{100\left|\mathcal{D}_k(\int_{\mathcal{D}_1}\text{dist}(w_k(x),K_xdx)^2\right|}\right\},
	\end{align*}
	and
	\begin{align}
		d(w^{(i_k,k)},w_k)<2^{-(k+1)}\varepsilon_0.
	\end{align}
	\quad Set $w_{k+1}=w^{(i_k,k)}$, then $w_{k+1}$ satisfies
	\begin{align*}
		\int_{\mathcal{D}_{k+1}}\text{dist}(w_{k+1}(x),K_x)dx\leq\{2^{-k},\frac{1}{2}\int_{\mathcal{D}_1}\text{dist}(w_k(x),K_x)dx\},
	\end{align*}
	\begin{align*}
		\max_{1\leq j<k}\left|\int_{\mathcal{D}_j}(w_{k+1}-w_k)w_kdx\right|<\min\left\{2^{-k},\frac{1}{100\left|\mathcal{D}_k(\int_{\mathcal{D}_1}dist(w_k(x),K_{\rho(x),q(x)}dx)^2\right|}\right\},
	\end{align*}
	and
	\begin{align}
		d(w_{k+1},w_k)<2^{-(k+1)}\varepsilon_0.
	\end{align}
	Hence, $\{w_k\}$ is a Cauchy sequence in $X$, and $w_k$ converges to some $w\in X$. We now need to show that for fixed $j$, $\{w_k\}$ converges strongly in $L^2(\mathcal{D}_j)$.
	
	\quad For any $k\geq j\geq 1$, $\mathcal{D}_1\subset\mathcal{D}_j\subset\mathcal{D}_{k+1}$, one has
	\begin{align*}
		\int_{\mathcal{D}_j}\left|w_{k+1}-w_k\right|dx&\geq\int_{\mathcal{D}_j}\text{dist}(w_k,K_x)dx-\int_{\mathcal{D}_j}\text{dist}(w_{k+1},K_x)dx\\
		&\geq\int_{\mathcal{D}_j}\text{dist}(w_k,K_x)dx-\int_{\mathcal{D}_{k+1}}\text{dist}(w_{k+1},K_x)dx\\
		&\geq\frac{1}{2}\int_{\mathcal{D}_j}\text{dist}(w_k,K_x)dx.
	\end{align*}
	Applying Holder's inequality yields
	\begin{align*}
		\|w_{k+1}-w_k\|^2_{L^2(\mathcal{D}_j)}\geq\frac{1}{4\left|\mathcal{D}_j\right|}\left(\int_{\mathcal{D}_j}\left|w_{k+1}-w_k\right|dx\right)^2.
	\end{align*}
	Hence,
	\begin{align*}
		\|w_{k+1}\|^2_{L^2(\mathcal{D}_j)}-\|w_k\|^2_{L^2(\mathcal{D}_j)}=\|w_{k+1}-w_k\|^2_{L^2(\mathcal{D}_j)}-2\int_{\mathcal{D}_j}(w_k-w_{k+1})w_kdx\geq 0.
	\end{align*}
	Thus, $\{\left|w_k\right|_{L^2(\mathcal{D}_j)}^2\}_{k=j}^{\infty}$ is non-decreasing, and it is convergent since it is bounded.
	
	\quad For $k>m\geq j$, we have
	\begin{align*}
		\|w_k-w_m\|_{L^2(\mathcal{D}_j)}^2&\leq 2(\|w_k-w_{k-1}\|^2+...+\|w_{m+1}-w_m\|^2) \\
		&=2\mathop{\Sigma}\limits_{l=m}\limits^{k-1}\left(\|w_{k+1}\|^2-\|w_l\|^2-2\int_{\mathcal{D}_j}(w_{l+1}-w_l)w_ldx\right) \\
		&\leq 2(\|w_k\|^2-\|w_m\|^2)+2^{-m+3}.
	\end{align*}
	Hence, $\{w_k\}$ also converges strongly in $L^2(\mathcal{D}_j)$.
\end{proof}

\section{Construction of suitable strict subsolutions}

\quad In this section, we construct suitable strict subsolutions to \eqref{eq2}. We need strict subsolutions with periodic boundary and also strict subsolutions with compact suppports. Though in general case, it is not easy to constuct subsolutions with compact supports since the solution of the Poisson equation $\Delta u=p$ may not have compact support if $p$ does not have compact support. Fortunately, Akramov and Wiedemann found that in some specific case of $p$, there exist compact support solutions of the Poisson equation\cite{MR0493420}, which helps on constructing strict subsolutions to \eqref{eq2}. 

\quad For any set $\Omega \subset \mathbb{R}^n$ and $\varepsilon>0$, define $\Omega^\varepsilon:=\{x\in\mathbb{R}^n:\text{dist}(x,\Omega)<\varepsilon\}$. Let $\omega:\mathbb{R}^n\rightarrow\mathbb{R}$ be a smooth spherically symmetric function such that: $w(x)$ is constant for $\left|x\right|<\frac{1}{2}$, $\text{supp}(\omega)\subset B_1(0)$, $\omega(w)\geq 0$ and
\begin{align*}
	\int_{\mathbb{R}^n}\omega(x)dx=1.
\end{align*}
\quad Denote $\omega^{\varepsilon}(x)=\frac{1}{\varepsilon^n}\omega(\frac{x}{\varepsilon})$ for $\varepsilon>0$. Since $\rho(x)=\bar{\rho}$ for $x\notin \Omega$, then denote $p(\bar{\rho})=\bar{p}$. Let $p_1(x)=p(\rho(x))-\bar{p}$. Under this definition, we have $\text{supp}(p_1)\subset\Omega$ and
\begin{align*}
	\int_{\mathbb{R}^n}p_1(x)dx=\int_{\Omega}p_1(x)dx=0.
\end{align*}

\begin{proposition}\label{pp4.1}(\cite[Lemma 5]{MR0493420}) Let $p^{\varepsilon}(x)=p_1(x)-p_1\ast\omega^{\varepsilon}(x)$, so that $\text{supp}(p^{\epsilon})\subset\Omega^{\varepsilon}$. Then there exists $u\in C_c^{2,\alpha}(\mathbb{R}^n)$ for every $0<\alpha<1$ such that
	\begin{align*}
		\Delta u=p^{\varepsilon}\quad \text{and}\quad \text{supp}(u)\subset\overline{\Omega^{\varepsilon}}.
	\end{align*}
\end{proposition}

\begin{proposition}\label{pp4.2}(\cite[Theorem III.3.1]{MR2808162}) Let $\tilde{\Omega}$ be a bounded locally Lipschitz domain containing $\overline{\Omega^{\epsilon}}\subset\mathbb{R}^n$. If $p\in C_c^{\infty}(\Omega^{\epsilon})$ satisfies the compatibility condition
	\begin{align*}
		\int_{\Omega'}p(x)dx=0,
	\end{align*}
	then, there exists $\{\phi_j\}_{j=1}^{n}\subset C_c^{\infty}(\tilde{\Omega})$ such that
	\begin{align*}
		\mathop{\sum}\limits_{j=1}^{n}\partial_j\phi_j=p.
	\end{align*}
\end{proposition}

\begin{proposition}\label{pp4.3}(\cite[Theorem 1]{MR1949165})
    Given $f\in L^n(\mathbb{T}^n;\mathbb{R})$ which satisfies $\int_{\mathbb{T}^n}fdx=0$. There exists a solution $u\in L^\infty \cap W^{1,n}(\mathbb{T}^n;\mathbb{R}^n)$ of
    \begin{align*}
        \text{div }u=f,
    \end{align*}
    moreover,
    \begin{align*}
        \|u\|_{L^\infty}+\|u\|_{W^{1,n}}\leq C(n)\|f\|_{L^n}.
    \end{align*}
\end{proposition}

\quad Now, we are ready for constructing strictly subsolutions to \eqref{eq2}.

\begin{proposition}\label{pp4.4} Let $\mathbf{B}=a\mathbf{I}_n$($a>0$), $p\in C_c^{\infty}(\Omega^{\varepsilon})$ such that $\int_{\Omega'}p(x)dx=0$, where $\tilde{\Omega}$ is a bounded locally Lipschitz domain and $\overline{\Omega^{\varepsilon}}\subset\mathbb{R}^n$. Then there exists a pair $(m,\mathbf{U})\in C_c^\infty(\Omega^\varepsilon)$ of a vector field and a symmetric trace-free matrix field satisfying
\begin{align*}
    \begin{cases}
        \text{div } m=0,\\
        \text{div } \mathbf{U}+\nabla p=\mathbf{B}m.
    \end{cases}
\end{align*}
\end{proposition}

\begin{proof}
	Define the matrix field $\mathbf{A}$ by
	\begin{align*}
		\mathbf{A}:=\left[\frac{n}{1-n}\left(\partial_i\phi_j-\frac{p}{n}\delta_{ij}\right)\right]_{ij}, \quad i,j=1,\dots,n,
	\end{align*}
	where $(\phi_j)_{j=1}^n$ are chosen as in Proposition \ref{pp4.2}. By direct computations,
	\begin{align*}
		\text{tr}(\mathbf{A})=\frac{n}{1-n}\left(\mathop{\sum}\limits_{i=1}^{n}\partial_i\phi_i-p\right)=\frac{n}{1-n}(p-p)=0.
	\end{align*}
	\quad Let $\mathbf{A}=\mathbf{U}+\mathbf{V}$, $\mathbf{U}=\frac{1}{2}(\mathbf{A}^t+\mathbf{A})$ and $\mathbf{V}=\frac{1}{2}(\mathbf{A}-\mathbf{A}^t)$. Then $\mathbf{U}\in \mathbb{S}_0^{n\times n}$, and $\mathbf{V}$ is a skew-symmetric matrix which satisfies $\mathbf{V}^t=-\mathbf{V}$. Hence
	\begin{align*}
		\text{div div } \mathbf{V}=\mathop{\sum}\limits_{i,j=1}^n\partial_i\partial_j V_{ij}=-\mathop{\sum}\limits_{i,j=1}^{n}\partial_j\partial_i V_{ij}.
	\end{align*}
	Therefore, $\text{div div }\mathbf{V}=0$.
	
	\quad Furthermore
	\begin{align*}
		\partial_j A_{ij}&=\left(\mathop{\sum}\limits_{j=1}^{n}\partial_j\partial_j\phi_j-\frac{\partial_i p}{n}\right)=\frac{n}{1-n}\partial_i\left(\mathop{\sum}\limits_{j=1}^{n}\partial_j\phi_j-\frac{p}{n}\right)\\
		&=\left(\frac{n}{1-n}\frac{n-1}{n}\right)\frac{\partial p}{\partial x_i}\\
		&=-\partial_i p.
	\end{align*}
	
	This gives $\text{div }\mathbf{A}=-\nabla p$. Set $m=-\mathbf{B}^{-1}\text{div }\mathbf{V}$. Then
	\begin{align*}
		\text{div }m=-\mathbf{B}^{-1}\text{div div } \mathbf{V}=0,
	\end{align*}
	and
	\begin{align*}
		\text{div }\mathbf{U}-\mathbf{B}m=\text{div }\mathbf{U}+\text{div }\mathbf{V}=\text{div }\mathbf{A}=-\nabla p.
	\end{align*}
	This finishes the proof of the proposition.
\end{proof}

\begin{proposition}\label{pp4.5} Let $n\geq 3$, $\mathbf{B}=a\mathbf{I}_n$, let $\rho\in C^1$ be a function such that $\rho>0$ on $\mathbb{R}^n$ and $\rho(x)=\bar{\rho}$ be constant on $\mathbb{R}^n\backslash \Omega$, and let $p(\rho)$ be a $C^1$ function such that $\int_{\Omega}p(\rho)dx=p(\bar{\rho})\left|\Omega\right|$.	Let $\tilde{\Omega}$ be a bounded locally Lipschitz domain with $\Omega\Subset\tilde{\Omega}$. Then, there exists $\tilde{\mathbf{U}}:\mathbb{R}^n\rightarrow \mathbb{S}_0^{n\times n}$, $\tilde{m}(x):\mathbb{R}^n\rightarrow \mathbb{R}^n$ and $\chi\in\mathbb{R}$ such that
	\begin{align*}
		&\text{div }\tilde{\mathbf{U}}+\nabla q=\mathbf{B}\tilde{m},\\
		&\text{supp}(\tilde{m}(x),\tilde{\mathbf{U}}(x))\subset \Omega',\\
		&\tilde{m}(x)\otimes \tilde{m}(x)-\rho(x)\tilde{\mathbf{U}}(x)<\rho(x)q(x)\mathbf{I}_n \text{ for all }x\in\mathbb{R}^n,
	\end{align*}
where $q(x)=p(\rho(x))+\frac{\chi}{n}$.
\end{proposition}
\begin{proof}
	Let $\varepsilon>0$ be so small such that $\overline{\Omega^\varepsilon}\subset\Omega'$ and $p^\epsilon=p(\rho)-p(\rho)* \omega^{\varepsilon}$. Then by Proposition \ref{pp4.1}, there exists $u\in C_c^{2,\alpha}(\tilde{\Omega})$ such that $\Delta u=p^{\varepsilon}$ and $\text{supp } u\subset\overline{\Omega^{\varepsilon}}$. Define a matrix by
	\begin{align*}
		\mathbf{U}_{ij}^{(1)}=-\frac{n}{n-1}\frac{\partial^2u}{\partial x_i\partial x_j}\ \text{ for }i\neq j, \quad
		\mathbf{U}_{ii}^{(1)}=-\frac{n}{n-1}\frac{\partial^2u}{\partial x_i^2}+\frac{p^{\varepsilon}}{n-1}.
	\end{align*}
	Under this definition, $\mathbf{U}^{(1)}$ is symmetric and
	\begin{align*}
		\text{tr}(\mathbf{U}^{(1)})=-\frac{n}{n-1}(\Delta u-p^{\varepsilon})=0.
	\end{align*}
	Furthermore, the straightforward computations give
	\begin{align*}
		-\text{div }U_i^{(1)}=\frac{n}{n-1}\frac{\partial^3u}{\partial x_i^3}-\frac{1}{n-1}\frac{\partial}{\partial x_i}p^{\varepsilon}+\frac{n}{n-1}\frac{\partial}{\partial x_i}\Delta u-\frac{n}{n-1}\frac{\partial^3u}{\partial x_i^3}=\frac{\partial}{\partial x_i}p^{\varepsilon}.
	\end{align*}
	Hence, we have $\text{div } \mathbf{U}^{(1)}=-\nabla p^{\varepsilon}$.
	
	\quad Since $p_1*\omega^\varepsilon\in C_c^\infty(\Omega^\varepsilon)$ satisfies
	\begin{align*}
		\int_{\mathbb{R}^n}p_1*\omega^{\varepsilon}(x)dx=0,
	\end{align*}
	it follows from Proposition \ref{pp4.4} that there exists a matrix field $U^{(2)}$ and a vector field $\tilde{m}$ such that $(\tilde{m},\mathbf{U}^{(2)})\in C^\infty(\bar{\Omega})$ satisfies
	\begin{align*}
	(\tilde{m},\mathbf{U}^{(2)})\in C_c^{\infty}(\Omega'),\quad \text{div }\tilde{m}=0,\quad \text{and}\quad
	\text{div }\mathbf{U}^{(2)}+\nabla(p_1*\omega^{\varepsilon})=\mathbf{B}\tilde{m}.
	\end{align*}
	\quad Define $\tilde{\mathbf{U}}:=\mathbf{U}^{(1)}+\mathbf{U}^{(2)}$, clearly, $\text{supp}(\tilde{m}(x),\tilde{\mathbf{U}}(x))\in \Omega'$. Furthermore,
	\begin{align*}
		\text{div }\tilde{\mathbf{U}}+\nabla q=\text{div}(\mathbf{U}^{(1)}+\mathbf{U}^{(2)})+\nabla q=-\nabla p^{\varepsilon}+\mathbf{B}\tilde{m}-\nabla p_1*w^{\varepsilon}+\nabla q=\mathbf{B}\tilde{m}.
	\end{align*}
    The proof of the proposition is completed.
\end{proof}

\section{Prove of Theorem 1.1}
\quad The proof of Theorem \ref{tm1.1} is divided into several parts. The starting points are the cases in periodic space.

\begin{proposition}\label{pp5.1}
    Let $n \geq 3$, $\rho\in C(\mathbb{T}^n;\mathbb{R}_{+})$, for any given continuous function $q=q(x)>0$, there exist infinitely many weak solutions $u\in L^\infty(\mathbb{T}^n;\mathbb{R}^n)$ to \eqref{eq1} satisfying $\left|u\right|^2=\dfrac{nq}{\rho}$.
\end{proposition}

\begin{proof}
    Since $\rho$ is a continuous positive function on $\mathbb{R}^{n}$, and $p$ is a continuous function of $\rho$. Define
	\begin{align*}
		X_0=\{w\in L^{\infty}_{w^\ast}: w \text{ is a subsolution such that } w(x)\in\mathcal{U}_{\rho,q}\quad \text{for all}\quad x\in\mathbb{R}^{n}\}.
	\end{align*}
    \quad Since $q$ is a continuous function on $\mathbb{R}^n$, if $w(x)=(m(x),\mathbf{U}(x))\in\mathcal{U}_{\rho,q}\subset K_{\rho,q}^{\text{co}}$, then
	\begin{align*}
		\left|m\right|^2\leq n\bar{\rho}\bar{q},\quad \left|\mathbf{U}\right|\leq 2n\bar{\rho}\bar{q},
	\end{align*}
    where $\bar{q}=\max q(x)$ and $\bar{\rho}=\max\rho(x)$. Therefore, $X_0$ is bounded in $L^2$. We define $X$ be the closure of $X_0$ in the weak $L^2$-topology. 
	
 \quad Next, one needs to construct a subsolution to the system \eqref{eq2}. Since $\rho\in C(\mathbb{T}^n;\mathbb{R}_{+})$, then $\int_{\mathbb{T}^n}\nabla p(x)dx=0$. By Proposition \ref{pp4.3}, for any $\partial_i p$, there exists an vector valued function $u^{(i)}$ satisfying $\text{div }u^{(i)}=\partial _i p$. Define the matrix valued function $\mathbf{U}$ as $\mathbf{U}_{ij}=-u_j^{(i)}$, then $\text{div }\mathbf{U}=-\nabla p$. Hence, $w_0=(0,\mathbf{U})$ is a strict subsolution to \eqref{eq2} and $w_0\in X_0$ by definition.
 
	\quad We denote the mapping $w\mapsto\int\left|w\right|^2dx$, which is a Baire-1 function in $X$. Hence, its continuity points form a residual set in $X$. By Lemma \ref{lm3.2}, after a rescaling, one can show that there exists a continuous strictly increasing function $\tilde{\Phi}$ with $\tilde{\Phi}(0)=0$ such that for any $w\in X_0$, there exists a sequence $w_k\in X_0$ such that
	\begin{align*}
		w_k\rightharpoonup w\quad \text{weakly in}\quad L^2 \quad \text{and}\quad
		\int\left|w_k-w\right|^2dx\geq\tilde{\Phi}(\int\text{dist}(w,K_{\rho,q})).
	\end{align*}
	\quad By diagonal procedure, continuity points of the mapping $w\mapsto\int\left|w\right|^2dx$ in $X$ are subsolutions $w$ such that $w(x)\in K_{\rho,q}$ for almost every $x\in \mathbb{R}^n$. The residual set in $X$ is dense, hence there exist infinitely many weak solutions to system \eqref{eq1} satisfying $\left|m\right|^2=n\rho q$.
\end{proof}

\quad Next, we consider the case for specific source terms and densities without periodic boundary conditions.
\begin{proposition}\label{pp5.2}
    Let $n\geq 3$, and $\mathbf{B}=a\mathbf{I}_n$ for some $a>0$, let $\Omega\subset\mathbb{R}^n$ be a bounded open set, $\rho\in C^1(\mathbb{R}^n;\mathbb{R}_{+})$ satisfies $\rho(x)=\bar{\rho}>0$ for $x\in \mathbb{R}^n\backslash\Omega$ and $p\in C^1(\mathbb{R}^n;\mathbb{R}_{+})$ satisfies $\int_{\Omega}p(\rho(x))dx=p(\bar{\rho})\left|\Omega\right|$. Then, there exists a function $q=q(x)$ such that \eqref{eq1} has infinitely many compact supported weak solutions $u\in L^\infty(\mathbb{R}^n;\mathbb{R}^n)$ satisfying $\left|u\right|^2=\dfrac{nq}{\rho}$.
\end{proposition}

\begin{proof}
    By assumption, Proposition \ref{pp4.5} implies there exists a strict subsolution $\tilde{\mathbf{U}}:\mathbb{R}^n\rightarrow \mathbb{S}_0^{n\times n}$, and $\tilde{m}(x):\mathbb{R}^n\rightarrow \mathbb{R}^n$ such that $\text{supp}(\tilde{m}(x),\tilde{\mathbf{U}}(x))\subset \Omega'$, and
    	\begin{align*}
		&\text{div }\tilde{\mathbf{U}}+\nabla q=\mathbf{B}\tilde{m},\\
		&\tilde{m}(x)\otimes \tilde{m}(x)-\rho(x)\tilde{\mathbf{U}}(x)<\rho(x)q(x)\mathbf{I}_n\quad \text{for all}\quad x\in\mathbb{R}^n.
	\end{align*}
	for some $\chi\in\mathbb{R}$. We choose $\chi$ big enough such that $(\tilde{m},\tilde{\mathbf{U}})$ is a strict subsolution to the system \eqref{eq2} with constraint set $K_{\rho,q}$ and $\text{supp}(\tilde{m},\tilde{\mathbf{U}})\subset\Omega'$. Hence, by Proposition \ref{pp3.1}, we can construct infinitely many pairs $(m,\mathbf{U})$ such that $(\rho,m,\mathbf{U},q)$ are subsolutions and $(m,U)(x)\in K_{\rho,q}$, and $\text{supp}(m-\tilde{m},\mathbf{U}-\tilde{\mathbf{U}})\subset\tilde{\Omega}$, which imply the existence of the compact supported solutions.
\end{proof}

\quad Finally, we consider the case for general source terms.

\begin{proposition}\label{pp5.3}
    Let $n\geq 3$, let $\Omega\subset\mathbb{R}^n$ be a bounded open set, $\rho\in C^0(\mathbb{R}^n;\mathbb{R}_{+})$ satisfies $\rho=\bar{\rho}>0$ for $x\in\mathbb{R}^n\backslash\Omega$ and $\rho<\bar{\rho}$ for $x\in\Omega$. Then, there exists a function $q=q(x)$ such that there exist infinitely many compact supported weak solutions $u\in L^\infty$ to \eqref{eq1} satisfying $\left|u\right|^2=\frac{nq}{\rho}$.
\end{proposition}

\begin{proof}
	Since $\rho=\rho(x)$ satisfies $\rho=\bar{\rho}>0$ for $x\in \mathbb{R}^n\backslash\Omega$, and $\rho<\bar{\rho}$ for $x\in \Omega$. Let $q(x):=\bar{\rho}-\rho(x)$, then $q(x)=0$ for $x\in \mathbb{R}^n\backslash \Omega$, and $q(x)>0$ for $x\in \Omega$. Let $\bar{w}=(0,0)$, then $(\rho,\bar{w},q)$ is a strict subsolution in $\Omega$. By Proposition \ref{pp3.1}, there exist infinitely many pairs $w=(m,\mathbf{U})$ such that $(\rho,w,q)$ are subsolutions and
	\begin{align*}
		(m,\mathbf{U})(x)\in K_{\rho,q}\quad \text{for a.e.}\quad x\in\bar{\Omega}, \text{ and supp}(m,\mathbf{U})\subset\bar{\Omega}
	\end{align*}
    Hence, there are infinitely many compact supported weak solutions to \eqref{eq1}.
\end{proof}

Combining Propositions \ref{pp5.1}, \ref{pp5.2} and \ref{pp5.3}, we complete the proof of Theorem \ref{tm1.1}.
 
\textbf{Acknowledgements}
 
The author is grateful to his advisor Chunjing Xie for stimulating discussions and for several detailed comments to improve the presentation of this paper.
\bibliographystyle{abbrv}
\bibliography{references}  
\end{document}